\documentclass[final]{siamltex}

\usepackage{graphicx} 
\usepackage{amsmath}
\usepackage{amsfonts}
\usepackage{url}
\usepackage{float}
\usepackage{cleveref}

\usepackage{adjustbox}

\usepackage[usenames,dvipsnames]{color}

\title{Closed form optimized transmission conditions for
  complex diffusion with many subdomains}

\author{V. Dolean\thanks{Department of Mathematics and Statistics, University of Strathclyde, Glasgow, UK, and Laboratoire J.A.~Dieudonn\'e, CNRS, University C\^ote d'Azur, Nice, France. E-mail: {work@victoritadolean.com}.}
\and M. J. Gander\thanks{University of Geneva}   
\and A. Kyriakis\thanks{Department of Mathematics and Statistics, University of Strathclyde, Glasgow, UK, E-mail: {Alexandros.Kyriakis@strath.ac.uk}.}
}

\begin{document}

\maketitle

\begin{abstract}
  Optimized transmission conditions in domain decomposition methods
  have been the focus of intensive research efforts over the past
  decade. Traditionally, transmission conditions are optimized for two
  subdomain model configurations, and then used in practice for many
  subdomains. We optimize here transmission conditions for the first
  time directly for many subdomains for a class of complex diffusion
  problems. Our asymptotic analysis leads to closed form optimized
  transmission conditions for many subdomains, and shows that the
  asymptotic best choice in the mesh size only differs from the two
  subdomain best choice in the constants, for which we derive the
  dependence on the number of subdomains explicitly, including the
  limiting case of an infinite number of subdomains, leading to new
  insight into scalability. Our results include both Robin and
  Ventcell transmission conditions, and we also optimize for the first
  time a two-sided Ventcell condition. We illustrate our results with
  numerical experiments, both for situations covered by our analysis
  and situations that go beyond.
\end{abstract}



\pagestyle{myheadings}
\thispagestyle{plain}
\markboth{V. Dolean, M.J. Gander, A. Kyriakis}{DD algorithms for the complex diffusion problem}

\section{Introduction}

Diffusion problems are ubiquitous in science and engineering. While
classical diffusion problems are real, there are also important
complex diffusion problems. For example in geophysics, the
magnetotelluric approximation of Maxwell's equations is a key tool to
extract information about the spatial variation of electrical
conductivity in the Earth's subsurface
\cite{vozoff1991magnetotelluric}. This approximation results in a
complex diffusion equation \cite{Donzelli:2019:ASM} of the form
\begin{equation}\label{eq:1}
  \Delta u - (\eta - i\varepsilon )u = f, \quad \mbox{in a domain $\Omega$},
\end{equation}
where $f$ is the source function, and $\eta$ and $\varepsilon$ are
strictly positive constants\footnote{In the magnetotelluric
  approximation we have $\eta=0$, but we consider the slightly more
  general complex diffusion case here. Note also that the zeroth order
  term in \eqref{eq:1} is much more benign than the zeroth order term
  of opposite sign in the Helmholtz equation, see
  e.g. \cite{Ernst:2012:WDS}.}.
  
We are interested here in designing and analyzing domain decomposition
methods for complex diffusion problems of the form \eqref{eq:1} that
use optimized transmission conditions; for an analysis of the
performance of the classical Schwarz method, see
\cite{donzelli2018schwarz}. Traditionally, such transmission
conditions are derived and optimized for simple two subdomain
configurations, like in optimized Schwarz methods, see
\cite{chen2021optimized, gander2020optimized, brunet2020natural,
  gander2019heterogeneous, dolean2018asymptotic, gander2017optimized,
  gander2016optimized2, gander2016optimizedH, bennequin2016optimized,
  gander2008schwarz,gander2006optimized} and references therein. We
investigate here the optimization of transmission conditions directly
for the many subdomain case, and also study the optimization problem
in the limit when the number of subdomains goes to infinity, using the
tool of limiting spectra, see \cite{Bootland:2021:APS} and the
references therein. Our analysis for overlapping methods shows that
the optimized transmission conditions in the many subdomain case
behave asymptotically when the overlap goes to zero like the optimized
transmission conditions for the two subdomain case, only the
constants, which we derive in closed form for a given number of
subdomains, differ. We optimize both Robin and Ventcell (second order)
transmission conditions, and for the first time also a two-sided
variant of the Ventcell conditions. Our results also hold for the
classical Laplace problem, by simple setting $\eta$ and $\varepsilon$
to zero. We furthermore get from our analysis a new scalability result
for complex diffusion problems, which was first observed in the
context of solvation models in \cite{Stamm3}, and then proved for the
Laplace problem in \cite{ciaramella2017analysis}, and also holds for
other domain decomposition methods in such strip type domain
decomposition settings \cite{Chaouqui:2018:OSC}, see also
\cite{ciaramella2018analysis} for maximum principle techniques and
\cite{ciaramella2018analysisIII} for Lions type projection arguments
in more general geometries. We do not consider decompositions with
cross-points here, for which new techniques that are just in
development now would be needed.

\section{Optimized Schwarz methods for many subdomains}\label{OSMSec}

Our preliminary results in the short conference proceedings paper
\cite{Dolean:2021:OTC} have shown that for Robin transmission
conditions in the magnetotelluric approximation of Maxwell's equation
the asymptotically optimal parameter choice for two and three
subdomains has the same dependence on the overlap parameter when it
becomes small. We then also explored this dependence with numerical
experiments for four, five and six subdomains and the asymptotic
dependence remained the same, only the constants seem to depend on the
number of subdomains.

We prove here that indeed the asymptotic dependence of the optimized
parameters is the same for any number of subdomains, and also derive
the precise constants which themselves have a clear dependence on the
number of subdomains. Using the technique of limiting spectra, we can
even prove this result when the number of subdomains goes to infinity.
We therefore have for the first time a formal proof that the classical
approach of optimizing transmission conditions for a two subdomain
model problem to obtain optimized Schwarz methods is fully justified
for their use on many subdomains. We also show this result for Ventcell
(second order) transmission conditions, and optimize for the first
time a two-sided Ventcell variant.

To study Optimized Schwarz Methods (OSMs) for \eqref{eq:1}, we use a
rectangular domain $\Omega$ given by the union of rectangular
  subdomains $\Omega_j:= ({a_j},{b_j}) \times (0,\hat L)$,
  $j=1,2,\ldots,J$, where $a_j=(j-1)L-\frac{\delta}{2}$ and
  $b_j=jL+\frac{\delta}{2}$, and $\delta$ is the overlap, like in
  \cite{Chaouqui:2018:OSC}. Our OSM computes for iteration index
  $n=1,2,\hdots$
\begin{equation} \label{eq:2}
  \begin{array}{rcll}
  \Delta u_j^{n} -(\eta-i \varepsilon) u_j^{n}&=&f&\mbox{in $\Omega_j$},\\
  -\partial_x u_j^{n}+p_j^{-} u_j^{n}&=&-\partial_x u_{j-1}^{n-1}+p_j^{-}u_{j-1}^{n-1}
  \quad &\mbox{at $x=a_j$}, \\
  \partial_x u_j^{n}+p_j^{+} u_j^{n}&=&\partial_x u_{j+1}^{n-1}+p_j^{+} u_{j+1}^{n-1}
  \quad &\mbox{at $x=b_j$},
\end{array} 
\end{equation}
where $p_j^{-}$ and $p_j^{+}$ are strictly positive parameters in the
so called two-sided OSM, see e.g. \cite{Gander:2007:OSM}, and we have at
the top and bottom homogeneous Dirichlet boundary conditions, and on
the left and right homogeneous Robin boundary conditions, i.e we put
for simplicity of notation $u_0^{n-1}=u_{J+1}^{n-1}=0$ in
\eqref{eq:2}. The Robin parameters are fixed at the domain boundaries
$x=a_1$ and $x=b_J$ to $p_1^-= p_a$ and $p_J^+= p_b$.  By linearity,
it suffices to study the homogeneous equations, $f=0$, and analyze
convergence to zero of the OSM \eqref{eq:2}.  Expanding the
homogeneous iterates in a Fourier series
$$
u_j^{n}(x,y)=\sum_{m=1}^{\infty}
v_j^{n}(x,\tilde{k})\sin(\tilde{k}y),
$$ where
$\tilde{k}=\frac{m\pi}{\hat{L}}$ to satisfy the homogeneous Dirichlet
boundary conditions at the top and bottom, we obtain for the Fourier
coefficients the equations
\begin{equation}\label{eq:6}
  \begin{array}{rcll}
  \partial_{xx}v_j^{n}-(\tilde{k}^{2}+\eta-i\varepsilon)
  v_j^{n}&=&0 & x\in  (a_j,b_j),  \\
  -\partial_x v_j^{n}+p_j^{-} v_j^{n}&=&
  -\partial_x v_{j-1}^{n-1}+p_j^{-} v_{j-1}^{n-1}\quad
  & \mbox{at $x=a_j$},     \\
  \partial_x v_j^{n}+p_j^{+} v_j^{n}&=&
  \partial_x v_{j+1}^{n-1}+p_j^{+} v_{j+1}^{n-1}\quad &
  \mbox{at $x=b_j$}.
\end{array} 
\end{equation}
The general solution of the differential equation is 
$$
  v_j^{n}(x,\tilde{k})=c_j e^{-\lambda(\tilde{k})x}
    +d_j e^{\lambda(\tilde{k})x},
    $$  
    where
$\lambda=\lambda(\tilde{k})=\sqrt{\tilde{k}^{2}+\eta-i\varepsilon}$.
We next define the Robin traces,
\begin{eqnarray*}
\mathcal{R}_{-}^{n-1}(a_j,\tilde{k})&:=& -\partial_x
v_{j-1}^{n-1}(a_{j},\tilde{k})+p_j^{-}v_{j-1}^{n-1}(a_{j},\tilde{k}),\\
\mathcal{R}_{+}^{n-1}(b_j,\tilde{k})&:=& \partial_x
v_{j+1}^{n-1}(b_j,\tilde{k})+p_j^{+}v_{j+1}^{n-1}(b_j,\tilde{k}).
\end{eqnarray*}
Inserting the solution into the transmission conditions in
\eqref{eq:6}, we obtain for the remaining coefficients $c_j$
and $d_j$ the linear system
\begin{align*}
  c_j{e^{ - \lambda {a_j}}}(p_j^{-} + \lambda ) + d_j{e^{\lambda {a_j}}}(p_j^{-} - \lambda ) &=\mathcal{R}_{-}^{n-1}(a_j,\tilde{k}), \\
c_j{e^{ - \lambda {b_j}}}(p_j^{+} - \lambda ) + d_j{e^{\lambda {b_j}}}(p_j^{+} + \lambda ) &= \mathcal{R}_{+}^{n-1}(b_j,\tilde{k}),
\end{align*}
whose solution is
\begin{align*}
  c_j &= \frac{1}{D_j}({e^{\lambda{b_j}}}(p_j^{+} + \lambda )
  \mathcal{R}_{-}^{n-1}(a_j,\tilde{k})
  -e^{\lambda{a_j}}(p_j^{-} - \lambda ){\mathcal{R}_{+}^{n-1}(b_j,\tilde{k})}),
  \tag{8} \label{eq:8}\\
  d_j &= \frac{1}{D_j}( - {e^{ - \lambda {b_j}}}(p_j^{+}- \lambda)
  \mathcal{R}_{-}^{n-1}(a_j,\tilde{k}) + {e^{ - \lambda {a_j}}}(p_j^{-} + \lambda )
  \mathcal{R}_{+}^{n-1}(b_j,\tilde{k})), \tag{9} \label{eq:9}
\end{align*}
where $ D_j:=(\lambda+p_j^{+})(\lambda+p_j^{-})e^{\lambda (L + \delta
  )} - (\lambda-p_j^{+})(\lambda-p_j^{-})e^{ - \lambda (L + \delta
  )}$.  We thus arrive for the Robin traces in the OSM at the
iteration formula
\begin{align*}
{\mathcal{R}_{-}^{n}(a_j,\tilde{k})} 
&= \alpha_j^-\mathcal{R}_{-}^{n - 1}({a_{j-1}},\tilde{k})
+\beta_j^-\mathcal{R}_{+}^{n - 1}({b_{j-1}},\tilde{k}), \, j=2,\ldots,J,\\
{\mathcal{R}_{+}^{n}(b_j,\tilde{k})} 
 &= \beta_j^+\mathcal{R}_{-}^{n - 1}({a_{j+1}},\tilde{k})
 +\alpha_j^+ \mathcal{R}_{+}^{n - 1}({b_{j+1}},\tilde{k}),  \, j=1,\ldots,J-1, 
\end{align*}
where 
\begin{equation}
\label{eq:alphabeta}
\begin{array}{rcll}
\alpha_j^-&:=&\frac{(\lambda+p_{j-1}^{+})(\lambda + p_j^{-})e^{\lambda \delta}-(\lambda-p_{j-1}^{+})(\lambda-p_j^{-})e^{-\lambda \delta}}
{(\lambda\!+\!p_{j-1}^{+})(\lambda\!+\!p_{j-1}^{-})e^{\lambda (L + \delta )}\!-\! (\lambda\!-\!p_{j-1}^{+})(\lambda\!-\!p_{j-1}^{-})e^{ - \lambda (L + \delta )}},& j=2,\ldots,J, \\
\alpha_j^+ &:=&\frac{(\lambda+p_{j+1}^{-})(\lambda + p_j^{+})e^{\lambda \delta}-(\lambda-p_{j+1}^{-})(\lambda-p_j^{+})e^{-\lambda \delta}}
{(\lambda\!+\!p_{j+1}^{+})(\lambda\!+\!p_{j+1}^{-})e^{\lambda (L + \delta )}\!-\! (\lambda\!-\!p_{j+1}^{+})(\lambda\!-\!p_{j+1}^{-})e^{ - \lambda (L + \delta )}},& j=1,\ldots,J-1,\\
\beta_j^{-}&:=&\frac{(\lambda + p_j^-)(\lambda - p_{j-1}^-) e^{-\lambda L}-(\lambda -p_j^-)(\lambda + p_{j-1}^-) e^{\lambda L}}
{(\lambda\!+\!p_{j-1}^{+})(\lambda\!+\!p_{j-1}^{-})e^{\lambda (L + \delta )}\!-\! (\lambda\!-\!p_{j-1}^{+})(\lambda\!-\!p_{j-1}^{-})e^{ - \lambda (L + \delta )}},& j=2,\ldots,J,\\
\beta_j^{+}&:=&\frac{(\lambda + p_j^+)(\lambda - p_{j+1}^+) e^{-\lambda L}-(\lambda -p_j^+)(\lambda + p_{j+1}^+) e^{\lambda L}}
{(\lambda\!+\!p_{j+1}^{+})(\lambda\!+\!p_{j+1}^{-})e^{\lambda (L + \delta )}\!-\! (\lambda\!-\!p_{j+1}^{+})(\lambda\!-\!p_{j+1}^{-})e^{ - \lambda (L + \delta )}},& j=1,\ldots,J-1.
\end{array}
\end{equation}
Defining the $2\times 2$ matrices
$$
  T_j^1:=\left[ {\begin{array}{cc}
    \alpha_j^-&\beta_j^{-}\\
    0&0
  \end{array}} \right], \, j=2,..,J \quad \mbox{and}\quad 
  T_j^2:=\left[ {\begin{array}{cc}
  0&0\\
\beta_j^{+}&\alpha_j^+
\end{array}} \right],\, j=1,..,J-1,
$$
we can write the OSM in substructured form (keeping the
  first and last rows and columns to make the block structure
  appear), namely
\begin{equation}\label{SubstructuredForm}
\underbrace{\left[ {\begin{array}{c}
0\\
\mathcal{R}_{+}^n({b_1},\tilde{k})\\
\mathcal{R}_{-}^n({a_2},\tilde{k})\\
\mathcal{R}_{+}^n({b_2},\tilde{k})\\
 \vdots \\
\mathcal{R}_{-}^n({a_j},\tilde{k})\\
\mathcal{R}_{+}^n({b_j},\tilde{k})\\
 \vdots \\
\mathcal{R}_{-}^n({a_{N - 1}},\tilde{k})\\
\mathcal{R}_{+}^n({b_{N - 1}},\tilde{k})\\
\mathcal{R}_{-}^n({a_N},\tilde{k})\\
0
    \end{array}} \right] }_\text{${{{\cal R} }^n}$}=\underbrace{ \left[\arraycolsep0.5em {\begin{array}{ccccccc}
        \\[-0.2em]
 &T_1^2& & & & &  \\[0.7em]
T_2^1& &T_2^2& & & &  \\[0.7em]
 &\ddots& &\ddots& & &  \\[0.7em]
 & &T_j^1& &T_j^2& &  \\[0.7em]
 & & &\ddots& &\ddots&  \\[0.7em]
 & & & &T_{N-1}^1& &T_{N-1}^2 \\[0.7em]
 & & & & &T_N^1& \\[1em]
\end{array}} \right]}_\text{$T$}  \underbrace{\left[ {\begin{array}{*{20}{c}}
0\\
\mathcal{R}_{+}^{n - 1}({b_1},\tilde{k})\\
\mathcal{R}_{-}^{n - 1}({a_2},\tilde{k})\\
\mathcal{R}_{+}^{n - 1}({b_2},\tilde{k})\\
 \vdots \\
\mathcal{R}_{-}^{n - 1}({a_j},\tilde{k})\\
\mathcal{R}_{+}^{n - 1}({b_j},\tilde{k})\\
 \vdots \\
\mathcal{R}_{-}^{n - 1}({a_{N - 1}},\tilde{k})\\
\mathcal{R}_{+}^{n - 1}({b_{N - 1}},\tilde{k})\\
\mathcal{R}_{-}^{n - 1}({a_N},\tilde{k})\\
0
    \end{array}} \right]}_\text{${{{\cal R} }^{n-1}}$}.
\end{equation}
If the parameters $p_j^{\pm}$ are constant over all the interfaces,
and we eliminate the first and the last row and column of $T$, $T$
becomes a block Toeplitz matrix. The best choice of the parameters
minimizes the spectral radius $\rho(T)$ over a numerically relevant
range of frequencies $K:=[\tilde{k}_{\min},\tilde{k}_{\max}]$ with
$\tilde{k}_{\min}:=\frac{\pi}{\hat{L}}$ and
$\tilde{k}_{\max}:=\frac{M\pi}{\hat{L}}$, $M\sim\frac{1}{h}$, where
$h$ is the mesh size, and is thus solution of the min-max problem
  $$
  \min_{p_j^{\pm}}\max_{\tilde{k}\in
    K}|\rho(T(\tilde{k},p_j^{\pm}))|.
    $$ 
\begin{remark}
  This formulation is the most generic possible and the convergence
  factor in all the other particular cases can be derived from here:
  \begin{itemize}
    \item Dirichlet boundary conditions at $x=a_1$ and $x=b_J$ when
      $p_a,p_b\rightarrow \infty$.
    \item Dirichlet transmission conditions at the interfaces between
      subdomains when $p_j^-\rightarrow \infty,\, j=2,...,J$ and
      $p_j^+\rightarrow \infty,\, j=1,...,J-1$.
    \item The one dimensional case when $\lambda$ is replaced by
      $\lambda(0) = \sqrt{\eta - i\varepsilon}$.
  \end{itemize}
\end{remark}    

\section{Optimized Robin transmission conditions}\label{RobinSec}

We first state without proof the results obtained in the short
conference proceedings paper \cite{Dolean:2021:OTC} in the two
subdomain case before presenting our new results for the case of an
arbitrary number of subdomains.
\begin{theorem}[Two Subdomain Robin Optimization]
  Let $s:=\sqrt{\tilde{k}^2_{\min}+\eta - i\varepsilon}$, where
  the complex square root is taken with the positive real part, and
  let $K$ be the real constant
  \label{Theorem2sub}
  \begin{equation}
  \label{eq:K}
    K := \Re \frac{s ((p_b + s)(p_a + s)-(s - p_b)(s - p_a)
      e^{-4sL} )}{((s - p_a)e^{-2sL}+ s + p_a)((s - p_b)e^{-2sL} + s +
      p_b)}.
  \end{equation}
  Then for two subdomains with one sided Robin transmission
  conditions, $p_1^{+}=p_2^{-}=:p$, the asymptotically optimized
  parameter $p$ for small overlap $\delta$ and associated convergence
  factor are
  \begin{equation}\label{eq:p2dom}
    p\sim 2^{-1/3}K^{2/3} \delta^{-1/3}, \quad
    \rho = 1 - 2^{4/3}K^{1/3} \delta^{1/3} + {\cal O}(\delta^{2/3}).
  \end{equation} 
  For two-sided Robin transmission conditions, $p_1^{+}\ne p_2^{-}$,
  the asymptotically optimized parameters for small overlap $\delta$
  and associated convergence factor are
  \begin{equation}\label{eq:2p2dom}
    p_1^+\sim  2^{-2/5} K^{2/5} \delta^{-3/5},\, p_2^-\sim  2^{-4/5} K^{4/5} \delta^{-1/5},\, \rho= 1- 2^{4/5}K^{1/5}  \delta^{1/5} + {\cal O}(\delta^{2/5}).
\end{equation} 
\end{theorem}
When Dirichlet BCs are used at $x=a_1$ and $x=b_J$, i.e when $p_a$ and
$p_b$ tend to infinity, then the expression of the constant can be
further simplified to
$$
K:=\Re \frac{ s( e^{2sL} + 1)}{(e^{2sL} -1)}. 
$$
The results of this kind of optimization are illustrated in
\Cref{fig:2sdopt} where we see that optimal values are obtained when
the convergence factor equioscillates.
\begin{figure}
  \centering
  \includegraphics[width=0.42\textwidth]{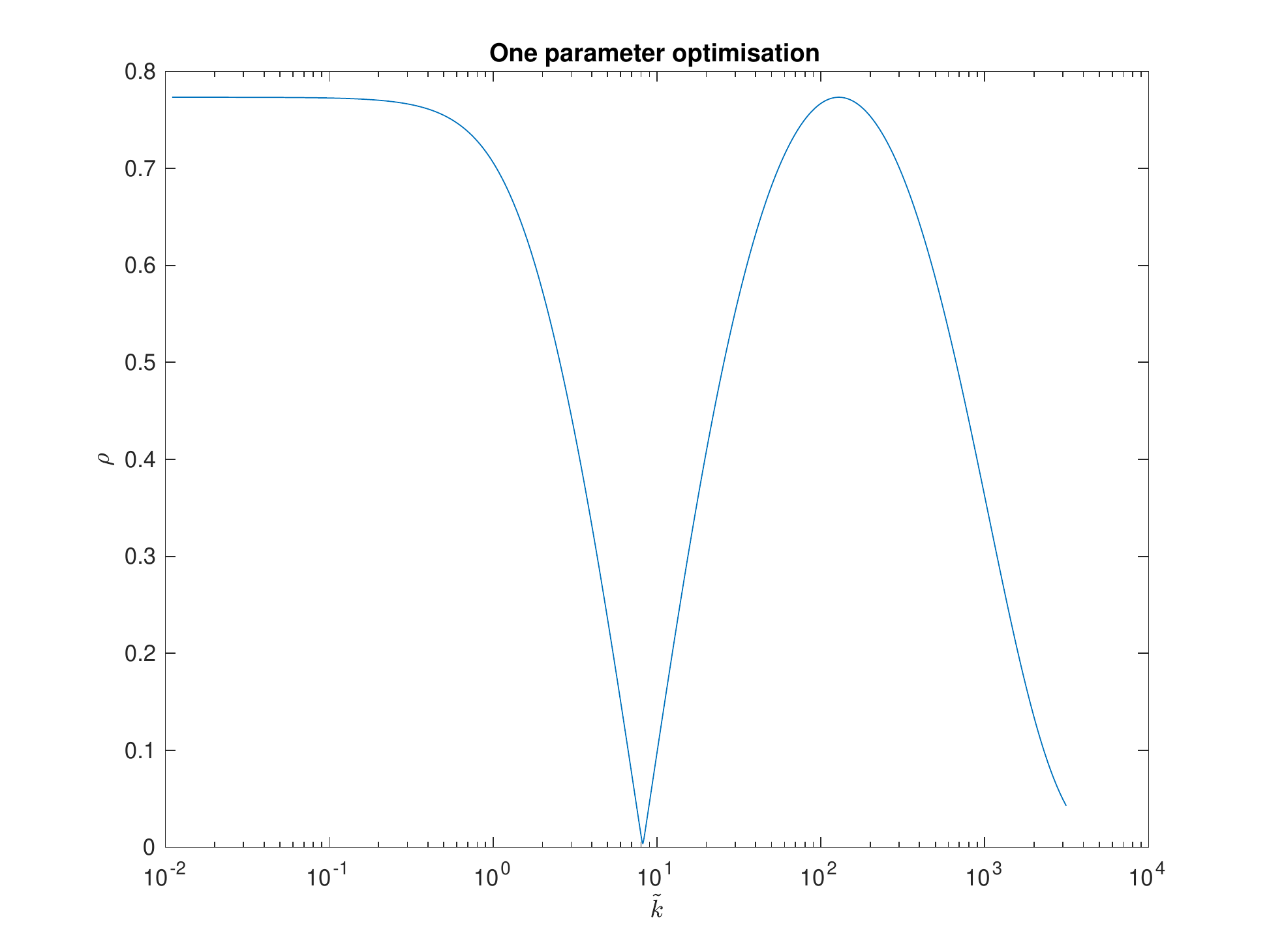}
  \includegraphics[width=0.42\textwidth]{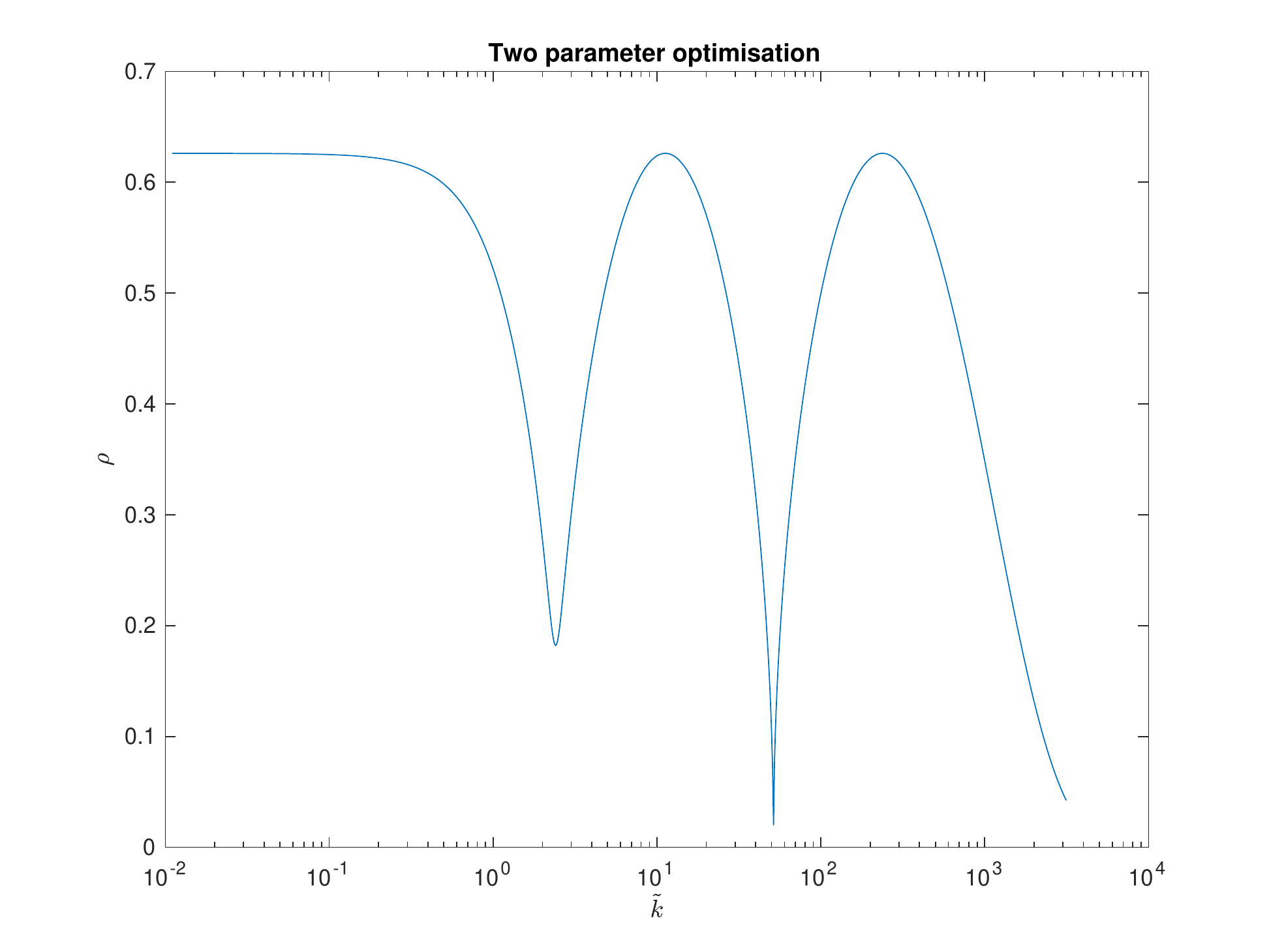}
  \caption{Equioscillation in numerical optimization with one and two
    optimized parameters.}
  \label{fig:2sdopt}
\end{figure}


\subsection{High frequency approximation of $T$ and $\rho$}

It is not possible to directly tackle the optimization of transmission
conditions for the many subdomain case, the spectral radius of the
iteration matrix $T$ in \eqref{SubstructuredForm} is too complex an
object. There is however an important observation: in the optimization
process, we see in Figure \ref{fig:2sdopt} that the convergence
factor, i.e. the spectral radius of $T$, equioscillates at different
frequency points $\tilde{k}$, and the local maximum points are for
$\tilde{k}$ large, which motivates the interest of the following Lemma
(see also \cite{gander2017optimized, chen2021optimized} for similar
high frequency approximations):
\begin{lemma}[High frequency approximation of $\rho$]
  For high frequencies, $\tilde{k}$ large, the convergence factor for
  $p_j^+=p_{j}^-=p$ behaves like
  \begin{equation}\label{rho1hf}
    \rho\sim \rho_{1,hf} = \left| \frac{\lambda-p}{\lambda+p} \right|
      e^{-\lambda \delta}, 
  \end{equation}
  and for $ p_j^+, \, p_{j+1}^-\in \{ p_1,p_2\}$ and $p_j^+ \ne
  p_{j+1}^-,\, j=1,\ldots,J-1$, it behaves like
  \begin{equation}\label{rho2hf}
  \rho^2\sim \rho^2_{2,hf} = \left| \frac{\lambda-p_1}{\lambda+p_1}\cdot
  \frac{\lambda-p_2}{\lambda+p_2} \right | e^{-2\lambda \delta}.
  \end{equation}
\end{lemma}
\begin{proof}
  When $\tilde{k}$ is large, the real part of $\lambda(\tilde{k})$ is
  large as well, and from \Cref{eq:alphabeta} we obtain because of the
  terms $e^{\lambda L}$ that for $\tilde{k}\rightarrow\infty$
$$
  \alpha_j^{\pm}(\tilde{k}) \sim 0,\quad
  \beta_j^-(\tilde{k})\sim-\frac{\lambda -p_j^-}{\lambda+p_{j-1}^{+}}
    e^{-\lambda \delta},\quad
  \beta_j^+(\tilde{k})\sim-\frac{\lambda -p_j^+}{\lambda\!+\!p_{j+1}^{-}} e^{-\lambda \delta}.
  $$
  The iteration matrix $T$ thus behaves for $\tilde{k}$ large like
\begin{equation}
  T\sim T_{hf}= \left[ \begin{array}{cccccccc}
0 & \beta_1^+ &0  & 0 & \hdots  & 0 & 0 \\
\beta_2^-& 0 & 0 & 0 & \hdots & 0 & 0  \\
0 & 0 & 0 & \beta_2^+ & \hdots & 0 & 0  \\
0 &0 &\beta_3^-&0 & \hdots & 0 & 0 \\
0 & 0&  \hdots & \ddots & \ddots & 0 & 0\\
0 & 0&  \hdots & 0 & 0 & 0 & \beta_{J-1}^+\\
0 &0 & \hdots & 0 & 0 & \beta_J^- & 0 
\end{array} \right].
\end{equation}
The eigenvalues of this matrix are given by the pairs $\pm
\sqrt{\beta_j^+\beta_{j+1}^-},\, j=1,..,J-1$ and therefore the high
frequency convergence factor is
$$
  \rho_{hf} = \max_j \left| \sqrt{\frac{(\lambda -p_j^-)(\lambda -p_j^+)}{(\lambda\!+\!p_{j-1}^{+})(\lambda\!+\!p_{j+1}^{-})}} \right | e^{-\lambda \delta},
$$
which leads to the result of the lemma.
\end{proof}

This result can also be understood intuitively: the coefficients
$\alpha_j^{\pm}$ relate interface values across subdomains, while the
coefficients $\beta_j^{\pm}$ relate interface values across the
overlap only, which is much smaller than the subdomain size. Since
high frequencies are damped rapidly in a diffusion problem over
spatial distance, only the terms $\beta_j^{\pm}$ related to the small
overlap remain relevant for the high frequency behavior of the
algorithm. This is why we see in \Cref{rho1hf} and \Cref{rho2hf} the
typical two subdomain convergence factors, see
e.g. \cite{gander2006optimized}, i.e in optimized Schwarz methods with
many subdomains, high frequencies still converge like if there were
only two subdomains.

\subsection{Optimization for $N$ subdomains}

The high frequency behavior of the convergence factor allows us to
study systematically the asymptotic form of the best parameter choice
for $N$ subdomains, depending on one remaining constant only:
\begin{lemma}[Generic optimized Robin asymptotics]\label{GenericLemma}
  The best choice in the one sided Robin transmission conditions is
  $p_j^+=p_{j}^-=p^*$, and when the overlap $\delta$ goes to zero, we
  have
  \begin{equation}
    p^*=\frac{C_k^2}{2} \delta^{-1/3}\quad\Longrightarrow\quad
    \rho^* \sim 1 - 2C_k \delta ^{1/3},
  \end{equation}
  where the constant $C_k$ depends on the number of subdomains. In the
  two-sided Robin transmission condition, the best choice is $ p_j^+,
  \, p_{j+1}^-\in \{ p_1^*,p_2^*\}$, $p_j^+ \ne p_{j+1}^-,\,
  j=1,\ldots,J-1$, with
  \begin{equation}
   p_1^* = C_{k_2}^2  \delta^{-3/5}\quad p_2^* = C_{k_2}^4\delta^{-1/5}
  \quad\Longrightarrow\quad
  \rho^* \sim 1 - 2C_{k_2} \delta ^{1/5}.
  \end{equation}
\end{lemma}
\begin{proof}
  In the one parameter case, we know from \cite{gander2006optimized}
  that for $\rho_{1,hf}$ from \Cref{rho1hf} the optimal parameter
  $p^*= C_p \delta^{-1/3}$, and a local maximum of $\rho_{1,hf}$ can
  be found at $k^* = C_k \delta^{-2/3}$. The relation between the two
  constants is $C_p = \frac{C_k^2}{2}$, as shown in detail in
  \cite[Proof of Theorem 1]{Dolean:2021:OTC}, and the maximum of
  the convergence factor is
  \begin{equation} \label{eq:rhohf}
     \rho(k^*) = 1 - 2C_k \delta ^{1/3} + {\cal O}(\delta^{2/3}),
  \end{equation}
  which proves the first claim.

  For two-sided Robin transmission conditions, the optimal parameters
  for the high frequency approximation $\rho_{2,hf}$ of the
  convergence factor from \Cref{rho2hf} were studied in
  \cite{gander2006optimized}, and they verify $p_1^*, p_2^*\in \{
  C_{k_2}^2 \delta^{-3/5}, \,C_{k_2}^4 \delta^{-1/5}\}$, with the
  corresponding maximum of the convergence factor given by
  \begin{equation}\label{eq:rhohf2}
    \rho(k^*) = 1 - 2C_{k_2} \delta ^{1/5}+ {\cal O}(\delta^{2/5}),
  \end{equation}
  which proves the second claim.
\end{proof}

It remains to study the constants $C_k$ and $C_{k_2}$, which are
determined by equioscillation with the low frequency convergence
factor, i.e. $\rho(\tilde k_{min})$, see Figure \ref{fig:2sdopt}, and
which depends on the number of subdomains, since we really need to
evaluate the spectral radius of the iteration matrix $T$ in
\eqref{SubstructuredForm}. To simplify the computations, we assume
Dirichlet boundary conditions at the outer boundaries of the global
domain, that is consider the limits when $p_a$ and $p_b$ go to
infinity.  We start by computing the leading order terms in $T$ for
small overlap $\delta$. For one sided Robin transmission conditions, $p
= C_p \delta^{-1/3} = \frac{C_k^2}{2} \delta^{-1/3}$, we obtain, with
$s:=\sqrt{\tilde{k}_{\min}^2+\eta - i\varepsilon}$,
$$
\begin{array}{c}
\displaystyle \alpha_j^+(\tilde{k}_{\min}) = \alpha_j^-(\tilde{k}_{\min}) = \frac{4 s e^{-sL}}{C_p (1- e^{-2sL} )} \delta^{1/3} =  \frac{8 s e^{-sL}}{C_k^2 (1- e^{-2sL} )} \delta^{1/3} := \tilde a,\\
\displaystyle \beta_j^+(\tilde{k}_{\min}) = \beta_j^-(\tilde{k}_{\min}) = 1 - \frac{2 s( e^{-2sL} + 1)}{C_p (1- e^{-2sL} )}\delta^{1/3} = 1 - \frac{4 s( e^{-2sL} + 1)}{C_k^2 (1- e^{-2sL} )}\delta^{1/3}=: \tilde b,
\end{array}
$$
which leads to the simplified low frequency iteration matrix
\begin{equation}
\label{eq:tlf1}
T_{lf,1 par}= \left[ \begin{array}{cccccccc}
0 & \tilde b & \tilde a & 0 & \hdots  & 0 & 0 \\
\tilde b & 0 & 0 & 0 & \hdots & 0 & 0  \\
0 & 0 & 0 & \tilde b & \hdots & 0 & 0  \\
0 & \tilde a &\tilde b&0 & \hdots & \tilde a & 0 \\
0 & 0&  \hdots & \ddots & \ddots & 0 & 0\\
0 & 0&  \hdots & 0 & 0 & 0 & \tilde b\\
0 &0 & \hdots & 0 & \tilde a & \tilde b & 0 
\end{array} \right].
\end{equation}
By computing the spectral radius of this matrix for $J=2,3,4,\ldots$
subdomains, we get for small overlap $\delta$
\begin{equation}\label{eq:const}
  \arraycolsep0.1em
\begin{array}{rcl}
  \rho_2(\tilde{k}_{\min}) &=& 1 - \frac{4}{C_k^2 }\Re \frac{s( e^{2sL} + 1)}{(e^{2sL} -1)}\delta^{1/3}, \\  
  \rho_3(\tilde{k}_{\min}) &=& 1 - \frac{4}{C_k^2 }\Re \frac{s( e^{2sL} + 1 -e^{sL} )}{(e^{2sL} -1)}\delta^{1/3}, \\
  \rho_4(\tilde{k}_{\min}) &=& 1 - \frac{4}{C_k^2 }\Re \frac{s( e^{2sL} + 1 -\sqrt{2}e^{sL} )}{(e^{2sL} -1)}\delta^{1/3},\\
  &\vdots&\\
  \rho_{J}(\tilde{k}_{\min}) &=& 1 - \frac{4}{C_k^2 } \Re \frac{ s( e^{2sL}+1 -
    2\cos\left(\frac{\pi}{J}\right) e^{sL})}{(e^{2sL} -1)}
  \delta^{1/3}.
\end{array}
\end{equation}
Now defining the new constant that appears,
\begin{equation}\label{KJ}
K_{J}:= \Re \frac{ s( e^{2sL}+1 - 2\cos\left(\frac{\pi}{J}\right) e^{sL})}{(e^{2sL} -1)},
\end{equation}
we obtain $\rho_{J}(\tilde{k}_{\min})\sim
1-\frac{4K_J}{C_k^2}\delta^{1/3}$, and equating this with the high
  frequency maximum $\rho(k^*)\sim 1 - 2C_k \delta ^{1/3}$ from
  \Cref{eq:rhohf} leads to
\begin{equation}\label{CK}
  C_k=(2K_J)^{1/3}.
\end{equation}

For two-sided Robin transmission conditions, $p_1,p_2 \in\{ C_{p_1}
\delta^{-1/5}, C_{p_2} \delta^{-3/5} \} = \{C_{k_2}^2 \delta^{-1/5} ,
C_{k_2}^4 \delta^{-1/5}\}$, we obtain
$$
\begin{array}{c}
\displaystyle \alpha_j^+(\tilde{k}_{\min}) = \alpha_j^-(\tilde{k}_{\min}) = \frac{2 s e^{-sL}}{C_{p_2} (1- e^{-2sL} )} \delta^{1/5} =  \frac{2 s e^{-sL}}{C_{k_2}^4 (1- e^{-2sL} )} \delta^{1/5} := \tilde a,\\
\displaystyle \beta_j^+(\tilde{k}_{\min}), \beta_{j+1}^-(\tilde{k}_{\min}) \in \{ \delta^{2/5} C_{k_2}^2 \tilde b, \frac{1}{\delta^{2/5} C_{k_2}^2} \tilde b\},\,  \tilde b =  1 - \frac{s( e^{-2sL} + 1)}{C_{k_2}^4 (1- e^{-2sL} )}\delta^{1/5},
\end{array}
$$
which leads to the low frequency iteration matrix
\begin{equation}
\label{eq:tlf2}
T_{lf,2par}= \left[ \begin{array}{cccccccc}
0 & \tilde b_+ & \tilde a & 0 & \hdots  & 0 & 0 \\
\tilde b_- & 0 & 0 & 0 & \hdots & 0 & 0  \\
0 & 0 & 0 & \tilde b_+ & \hdots & 0 & 0  \\
0 & \tilde a &\tilde b_- &0 & \hdots & \tilde a & 0 \\
0 & 0&  \hdots & \ddots & \ddots & 0 & 0\\
0 & 0&  \hdots & 0 & 0 & 0 & \tilde b_+\\
0 &0 & \hdots & 0 & \tilde a & \tilde b_- & 0 
\end{array} \right],
\end{equation}
where in fact the couple $ \tilde b_+ \ne \tilde b_-$ can vary along
the diagonal but always lays in the set $\{ \delta^{2/5} C_{k_2}^2
\tilde b, \frac{1}{\delta^{2/5} C_{k_2}^2} \tilde b\}$ which does not
change the eigenvalues of the matrix.  By computing the spectral
radius of this matrix for $J=2,3,4,\ldots$ subdomains we get for 
small overlap $\delta$
\begin{equation}
  \rho_J(k_{\min}) \sim 1 - \frac{K_J}{C_{k_2}^4}\delta^{1/5}
\end{equation}
with the same constant $K_J$ from \Cref{KJ}, and equating with 
$\rho(k^*) \sim 1 - 2C_{k_2} \delta ^{1/5}$ from \Cref{eq:rhohf2}, we obtain
\begin{equation}
  C_{k_2}=\frac{K_J}{2^{1/5}}.
\end{equation}
We therefore have, using Lemma \ref{GenericLemma}, the following
result for the $J$ subdomain decomposition:
\begin{theorem}[$J$ Subdomain Robin Optimization]\label{ThRobinJ}
  For $J$ subdomains and one sided Robin transmission conditions,
  $p_j^{+}=p_j^{-}$, the asymptotically optimized parameters for small
  overlap $\delta$ and associated convergence factor are
  \begin{equation}\label{eq:pJdom}
    p_j^{+}=p_j^{-}=p^*\sim \left(\frac{K_J^2}{2}\right)^{1/3}\delta^{-1/3}, \quad
    \rho \sim 1 - 2^{4/3}K_J^{1/3} \delta^{1/3},
  \end{equation}
  with the constant $K_J$ from \Cref{KJ}. For two-sided Robin
  transmission conditions, $p_j^{+}\ne p_j^{-}$, the asymptotically
  optimized parameters for small overlap $\delta$ and associated
  convergence factor are
  \begin{equation}\label{eq:2pJdom}
    p_j^+=p_+^*\sim  \left(\frac{K_J}{2^{1/5}}\right)^2 \delta^{-3/5},\quad
    p_j^-=p_-^*\sim  \left(\frac{K_J}{2^{1/5}}\right)^4 \delta^{-1/5},\quad
    \rho \sim 1- 2^{4/5}K_J  \delta^{1/5},
  \end{equation}
  and the role of $p_+^*$ and $p_-^*$ can be switched without changing the
  result.
\end{theorem}

\subsection{Optimization when $J$ goes to infinity}

We now use the limiting spectrum approach to study the optimized
parameters when the number of subdomains goes to infinity. To do so,
we must assume that the outer Robin boundary conditions use
the same optimized parameter as at the interfaces, in order to have
the Toeplitz structure needed for the limiting spectrum approach. We
state here without proof the result obtained in
\cite{Dolean:2021:OTC}.
\begin{theorem}[Infinite Number of Subdomains Robin Optimization]\label{ThRobinInf}
  With all Robin parameters equal, $p_j^{-}=p_j^{+}=p$, the
  convergence factor of the OSM satisfies the bound
  \[
    \rho = \mathop {\lim }\limits_{J \to  + \infty } \rho ({T_{2d}^{OS}}) \le \max \Big\{\left| {\alpha -\beta} \right|,\left| {\alpha + \beta} \right|\Big\}<1,
  \]
  where 
  $$
  \alpha  =\frac{(\lambda+p)^2 e^{\lambda
    \delta}-(\lambda-p)^2 e^{-\lambda \delta}}{(\lambda+p)^2e^{\lambda (L + \delta )}
  - (\lambda-p)^2 e^{ - \lambda (L + \delta )}},\,
  \beta =\frac{(\lambda - p)(\lambda + p) (e^{-\lambda L}-e^{\lambda L})}{(\lambda+p)(\lambda+p)e^{\lambda (L + \delta )} - (\lambda-p)(\lambda-p)e^{ - \lambda (L + \delta )}}.
  $$
  The asymptotically optimized parameter and associated
  convergence factor are
  \begin{equation}\label{eq:Jdom}
      p^*= 2^{-1/3} K_{\infty}^{2/3}\delta^{-1/3},\quad
      \rho=  1- 2^{4/3}K_{\infty}^{1/3}\delta^{1/3}+{\cal O}(\delta^{2/3}), 
  \end{equation}
  with the constant 
  \begin{equation}\label{eq:cnr}
  K_{\infty}:=\Re \frac{s(e^{sL}-1)}{e^{sL}+1}.
  \end{equation}
  If we allow two-sided Robin parameters, $p_j^{-}=p^-$ and
  $p_j^+=p^{+}$, the OSM convergence factor satisfies the bound
  \[
    \rho = \mathop {\lim }\limits_{J \to  + \infty } \rho ({T_{2d}^{OS}}) \le \max \Big\{\left| {\alpha -\sqrt{ \beta_{-} \beta_{+}}} \right|,\left| {\alpha + \sqrt{\beta_{-}\beta_{+}}} \right|\Big\}<1,
  \]
  where 
  $$
  \alpha = \frac{(\lambda+p^+)(\lambda+p^-) e^{\lambda
    \delta}-(\lambda-p^+) (\lambda-p^-)e^{-\lambda \delta}}{D},\, \beta^{\pm} =\frac{(\lambda^2 - (p^{\mp})^2)(e^{-\lambda L}-e^{\lambda L})}{D},
    $$
  with 
  $
  D=(\lambda+p^+)(\lambda+p^-)e^{\lambda (L + \delta )} - (\lambda-p^+)(\lambda-p^-)e^{ - \lambda (L + \delta )}.
  $
  The asymptotically optimized parameter choice $p_-^* \ne p_+^*$ 
  and the associated convergence factor are    
  $$
    p_-^*,p_+^* \in \left\{ K_{\infty}^{2/5}\delta^{-3/5},\, K_{\infty}^{4/5}\delta^{-1/5}  \right\},
   \quad \rho = 1- 2K_{\infty}^{1/5}\delta^{1/5} + {\cal O}(\delta^{2/5}),
   $$
   with the same constant $K_{\infty}$ as for the one sided case in
   \eqref{eq:cnr}.
\end{theorem}

Even though we had to use Robin outer boundary conditions to obtain
$K_\infty$, the constant $K_J$ we obtained for a finite number $J$ of
subdomains with Dirichlet boundary conditions converges when $J$
becomes large to $K_{\infty}$,
\begin{equation}\label{LimitKJ}
\lim_{J\to\infty}K_J=
  \lim_{J\to\infty}\Re \frac{ s( e^{2sL}+1 - 2\cos\left(\frac{\pi}{J}\right) e^{sL})}{(e^{2sL} -1)}=\Re\frac{s(e^{sL}-1)}{e^{sL}+1}=K_{\infty}.
\end{equation}
We show in Figures \ref{fig:constJ1} and \ref{fig:constJ2}
\begin{figure}
\includegraphics[width = 0.32\linewidth]{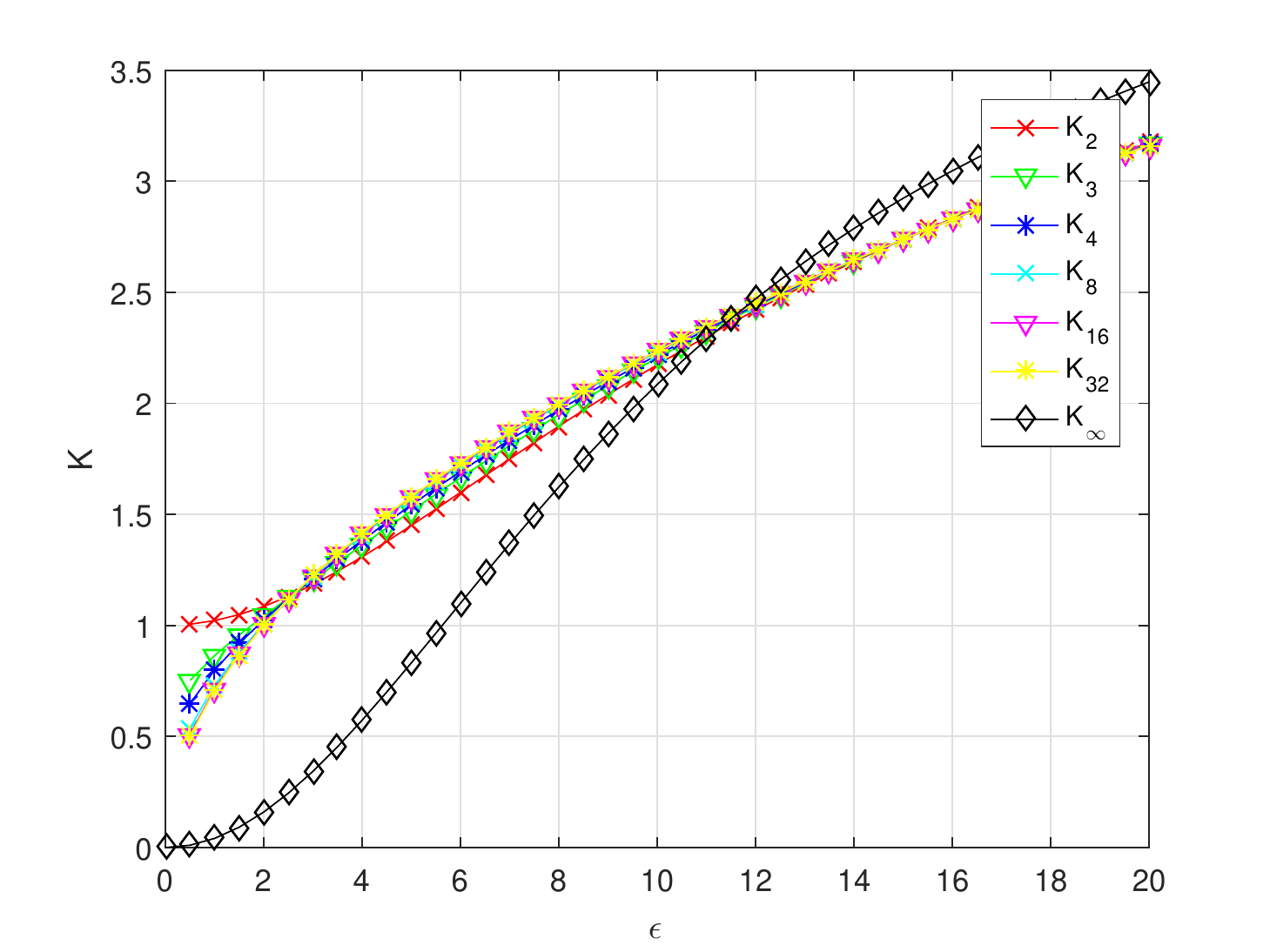}
\includegraphics[width = 0.32\linewidth]{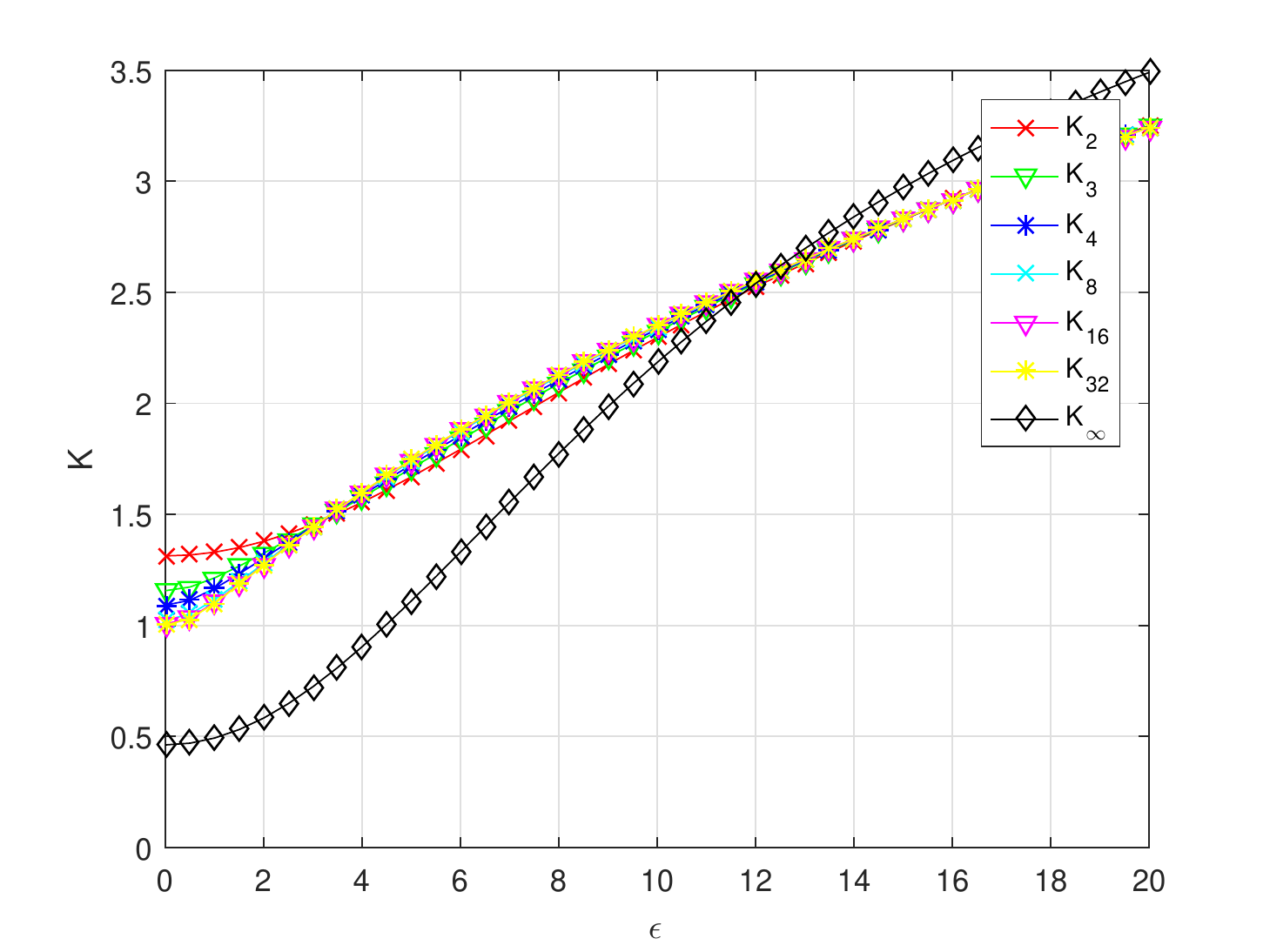}
\includegraphics[width = 0.32\linewidth]{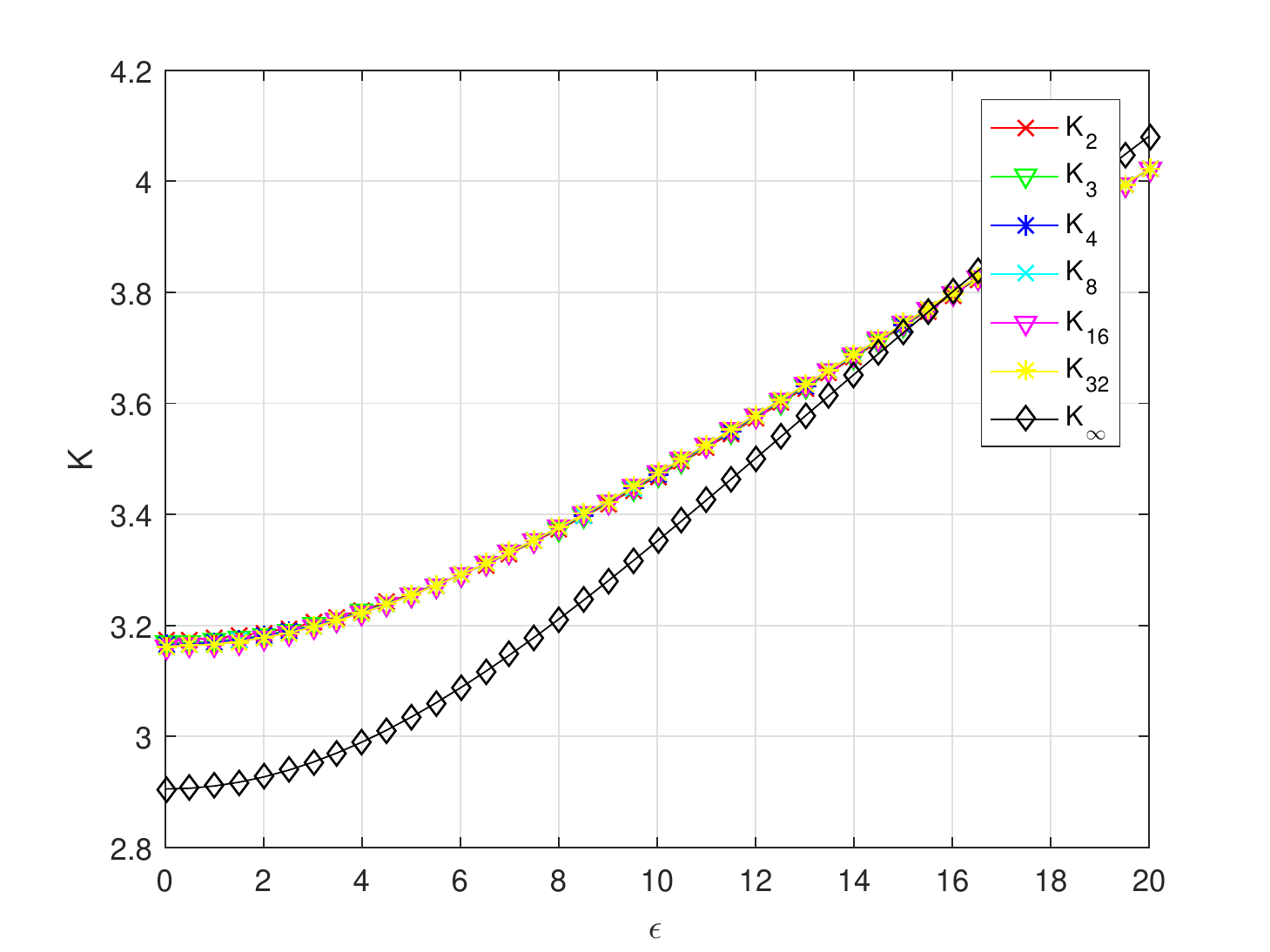}
\caption{Optimized constants for different numbers of subdomains for a
  fixed $\eta$ and $L=1$ as a function of $\varepsilon$.}
\label{fig:constJ1}
\end{figure}
\begin{figure}
\includegraphics[width = 0.32\linewidth]{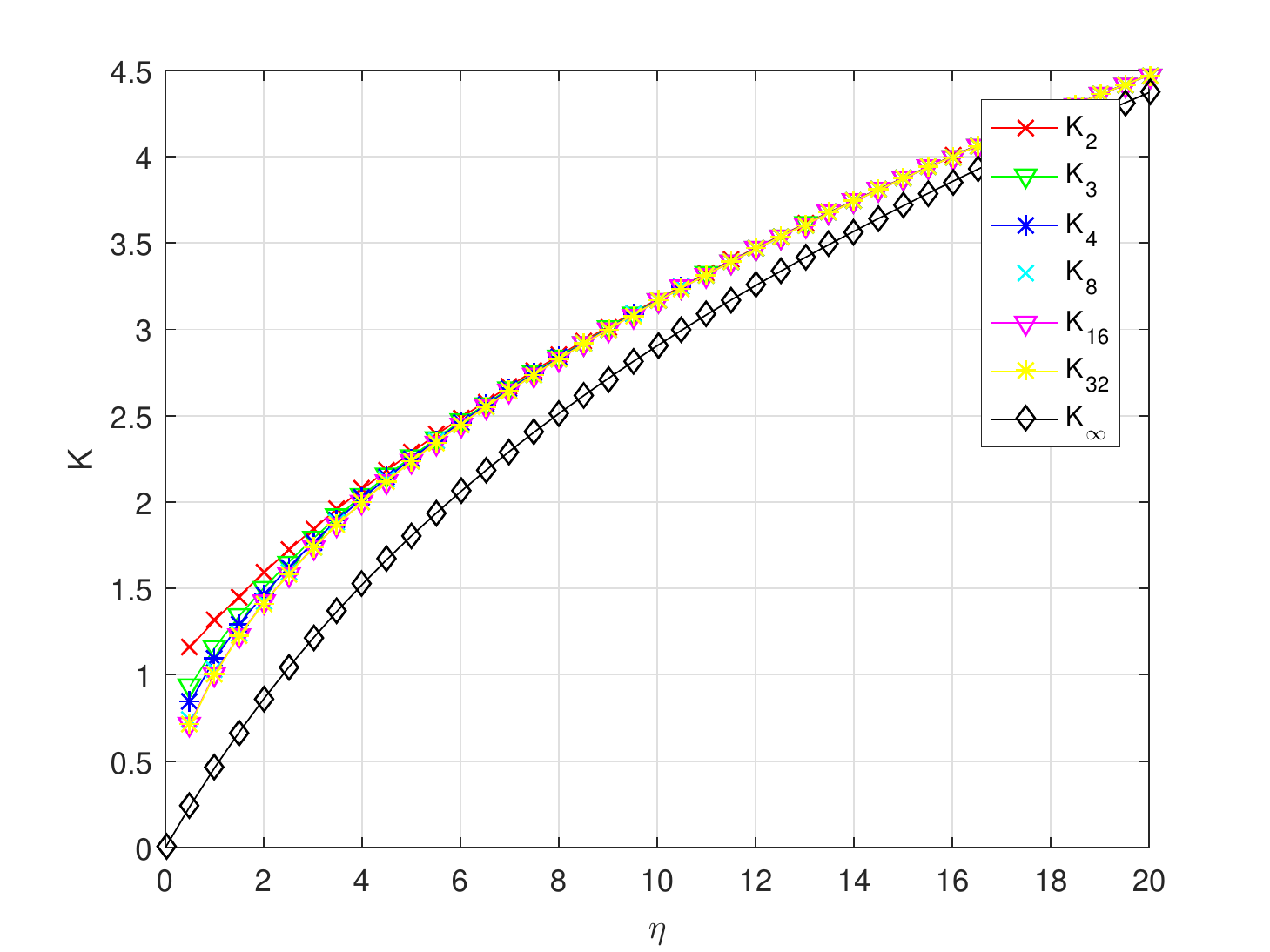}
\includegraphics[width = 0.32\linewidth]{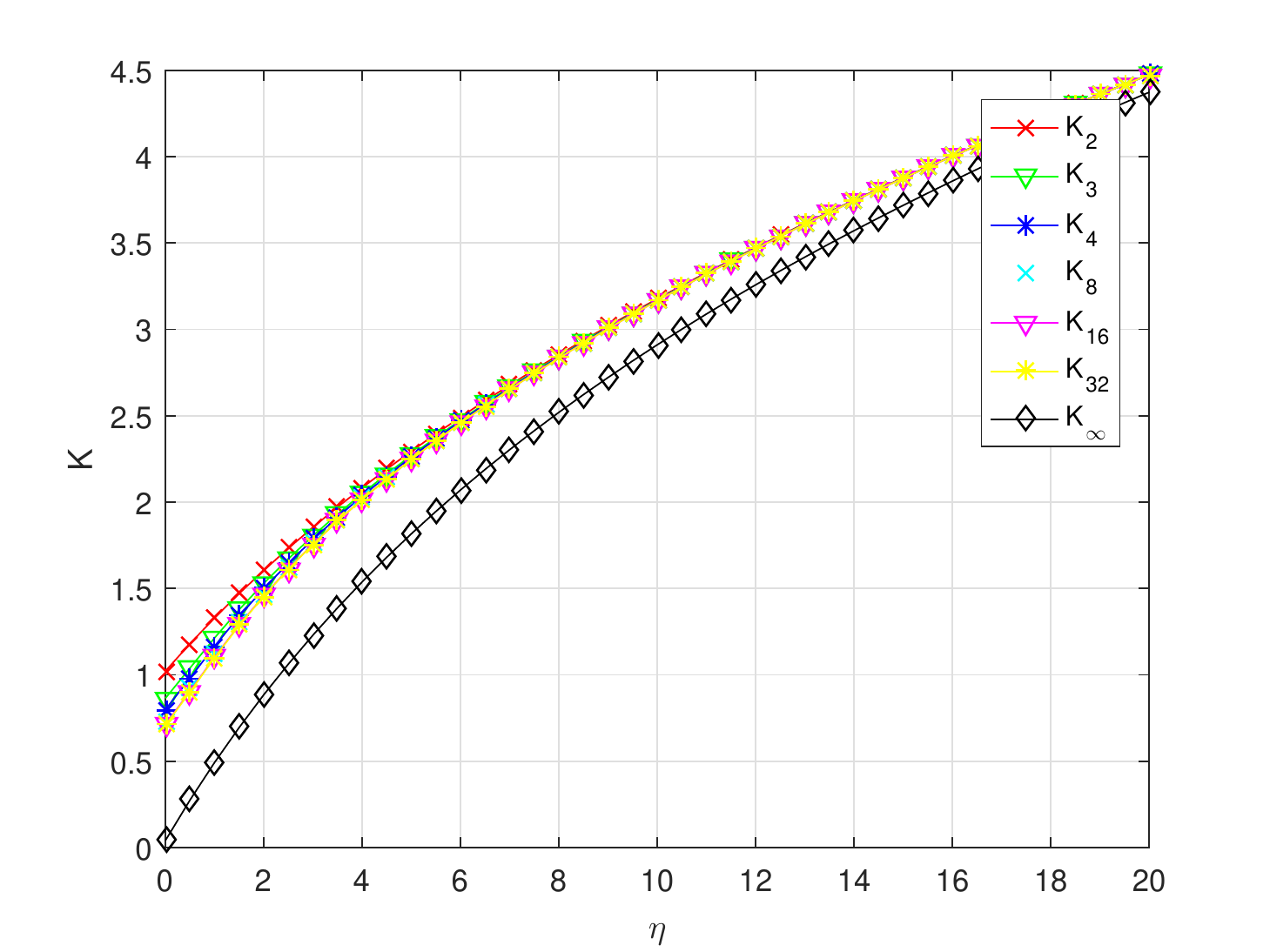}
\includegraphics[width = 0.32\linewidth]{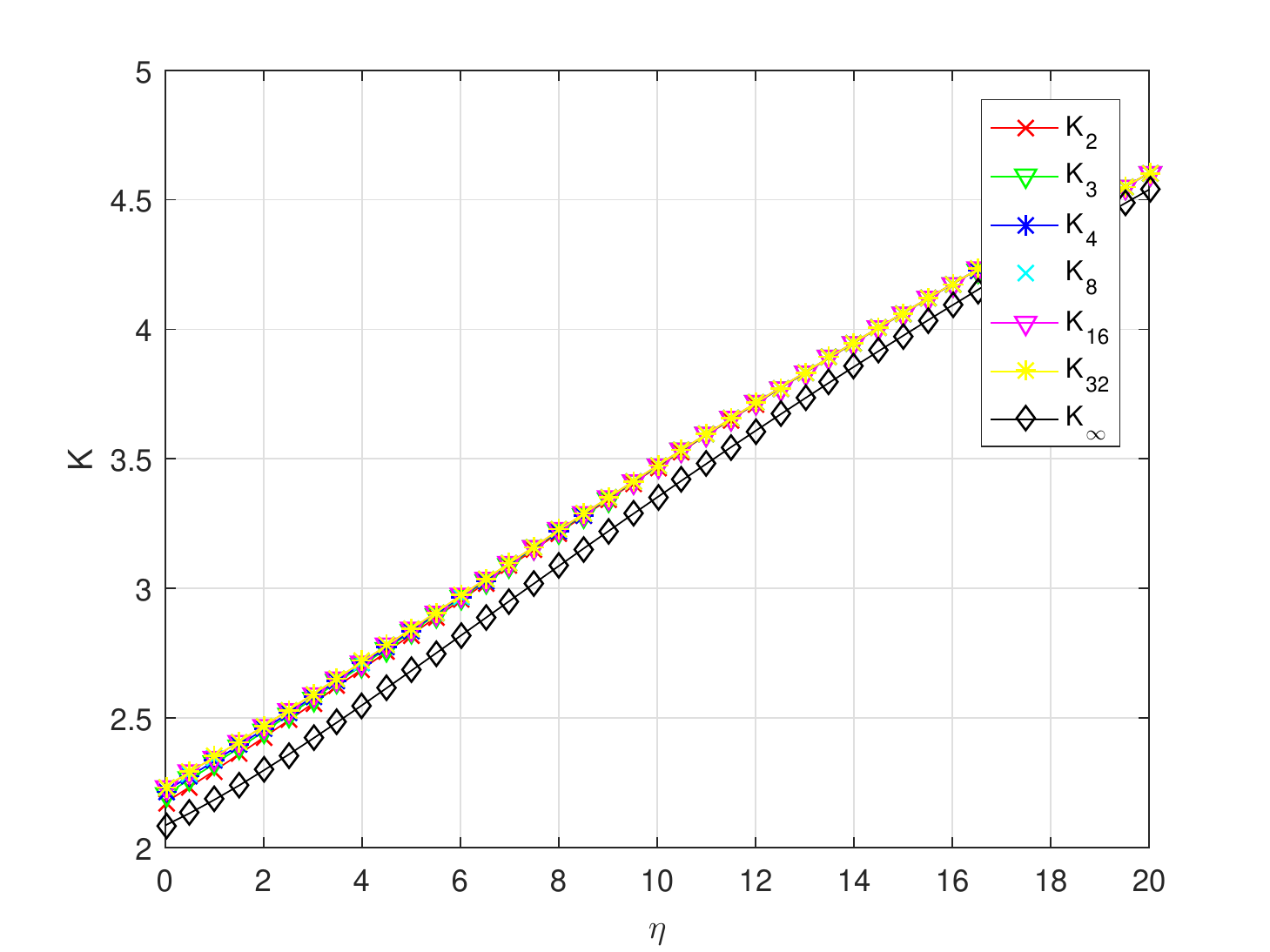}
\caption{Optimized constants for different numbers of subdomains for a
  fixed $\varepsilon$ and $L=1$ as a function of $\eta$.}
\label{fig:constJ2}
\end{figure}
how the constants $K_J$ evolve as functions of the problem parameters $\eta$ and $\varepsilon$ for different numbers of subdomains $J$ and
how they approach the limiting value $K_{\infty}$ as the number of
subdomains increases. We see that both the two subdomain optimization
and the limiting spectrum analysis result for an infinite number of
subdomains give quite good approximations for intermediate numbers of
subdomains over a large range of problem parameter values, the
specific optimization for a given number of subdomains only becomes
important when both $\eta$ and $\varepsilon$ are small.

\subsection{Scalability}\label{ScalabilitySubSec}

From the constant $K_J$ in \eqref{KJ} which governs the
  convergence factor in Theorem \ref{ThRobinJ}, and $K_{\infty}$ in
  \eqref{eq:cnr} which governs the convergence factor when the number
  of subdomains $J$ goes to infinity in Theorem \ref{ThRobinInf}, we
  see that for fixed subdomain width $L$, $K_J$ is robust when the
  number of subdomains $J$ increases, and thus our one-level methods
  are weakly scalable in this setting of strip decompositions, a
  result first proved for Laplace problems using different techniques
  in \cite{ciaramella2017analysis, ciaramella2018analysis,
    ciaramella2018analysisIII}. Furthermore, we have seen in
  \eqref{LimitKJ} that $K_J$ converges to $K_{\infty}$ when the number
  of subdomains $J$ becomes large, which shows that the methods are
  weakly scalable independently of the outer boundary conditions on
  the left and right of the strip decomposition. We will illustrate
  this scalability for fixed $L$ in Subsections
  \ref{NumSubSecScalability} and \ref{NumSubSecGMRES}, and also show
  that scalability is lost when the subdomain size $L$ becomes small
  when their number $J$ increases, as predicted by the formulas for
  $K_J$ in \eqref{KJ} and $K_{\infty}$ in \eqref{eq:cnr}.

\section{Optimized Ventcell transmission conditions}\label{VentcellSec}

To obtain Ventcell (second order) transmission conditions with even
better performance, we replace the coefficient $p_j^{\pm}$ from
\cref{eq:2} by second order differential operators along the
interface.  Writing again the local solutions as a Fourier series, the
Fourier coefficients satisfy the equations with Ventcell transmission
conditions,
\begin{equation}\label{eq:2nd}
  \begin{array}{rcll}
  \partial_{xx}v_j^{n}-(\tilde{k}^{2}+\eta-i\varepsilon)
  v_j^{n}&=&0 & x\in  (a_j,b_j),  \\
  -\partial_x v_j^{n}+(p_j^{-} + \tilde{k}^{2} q_j^{-}) v_j^{n}&=&
  -\partial_x v_{j-1}^{n-1}+(p_j^{-} + \tilde{k}^{2} q_j^{-}) v_{j-1}^{n-1}\quad
  & \mbox{at $x=a_j$},     \\
  \partial_x v_j^{n}+(p_j^{+} + \tilde{k}^{2} q_j^{+}) v_j^{n}&=&
  \partial_x v_{j+1}^{n-1}+(p_j^{+} + \tilde{k}^{2} q_j^{+}) v_{j+1}^{n-1}\quad &
  \mbox{at $x=b_j$}.
\end{array} 
\end{equation}
As in Section \ref{RobinSec}, the interface iteration involves a
block-Toeplitz iteration matrix which can be obtained by replacing
$p_j^{\pm} $ by $p_j^{\pm} + \tilde{k}^{2} q_j^{\pm}$ in the matrix
from \cref{SubstructuredForm}. The new iteration matrix depends on two
sets of parameters, $T_2(\tilde k, p_j^{\pm},q_j^{\pm}):= T(\tilde k,
p_j^{\pm}+\tilde k^2 q_j^{\pm})$, and we need to solve now the min-max
problem
$$
\min_{p_j^{\pm},q_j^{\pm}} \max_{\tilde k \in [\tilde k_{min},\tilde k_{max}]} |\rho(T_2(\tilde k, p_j^{\pm},q_j^{\pm}))|.
$$
Like in the case of Robin transmission conditions from Theorem
\ref{Theorem2sub}, we start by showing an optimization result in the
case of two subdomains, and then generalize this to the case of many
subdomains. We will see that the optimized parameters in the Ventcell
case depend on the same constants as for the Robin case.
\begin{theorem}[Two Subdomain Ventcell Optimization]
  Let $s:=\sqrt{\tilde{k}^2_{\min}+\eta - i\varepsilon}$, where
  the complex square root is taken with the positive real part, and
  let $K$ be the real constant given in \eqref{eq:K}.  Then for two
  subdomains with one sided Ventcell transmission conditions,
  $p_1^{+}=p_2^{-}=:p$, $q_1^{+}=q_2^{-}=:q$, the asymptotically
  optimized parameters $p$ and $q$ for small overlap $\delta$ and
  associated convergence factor are
    \label{Theorem2sub2par}
  \begin{equation}\label{eq:p2dom2par}
     p\sim 2^{-3/5}K^{4/5} \delta^{-1/5}, \, q \sim 2^{-1/5}K^{-2/5} \delta^{3/5}, \quad
    \rho \sim 1 -  2^{8/5}K^{1/5} \delta^{1/5} + {\cal O}(\delta^{2/5}).
  \end{equation} 
  For two-sided Ventcell transmission conditions, $p_1^{+} \ne
  p_2^{-}$, $q_1^{+}\ne q_2^{-}$, the asymptotically optimized
  parameters for small overlap $\delta$ are
  \begin{equation}\label{eq:2p2dom2par}
    p_1^+ \sim  2^{-8/9} K^{8/9} \delta^{-1/9},\, q_1^+ \sim 2^{2/9} K^{-2/9} \delta^{7/9},\, p_2^-\sim  2^{-2/3} K^{2/3} \delta^{-1/3},\, q_2^{-} \sim 2^{4/9} K^{-4/9} \delta^{5/9}
  \end{equation} 
  and the associated convergence factor is $\rho \sim 1-
  2^{8/9}K^{1/9} \delta^{1/9} + {\cal O}(\delta^{2/9}).$
\end{theorem}
\begin{proof}
The proof follows the lines of the proof of Theorem \ref{Theorem2sub}
from \cite{Dolean:2021:OTC} and uses the fact that the solution of the
min-max problem equioscillates, see Figure \ref{fig:2sdopt2nd}.
\begin{figure}[t]
  \centering
  \includegraphics[width=0.42\textwidth]{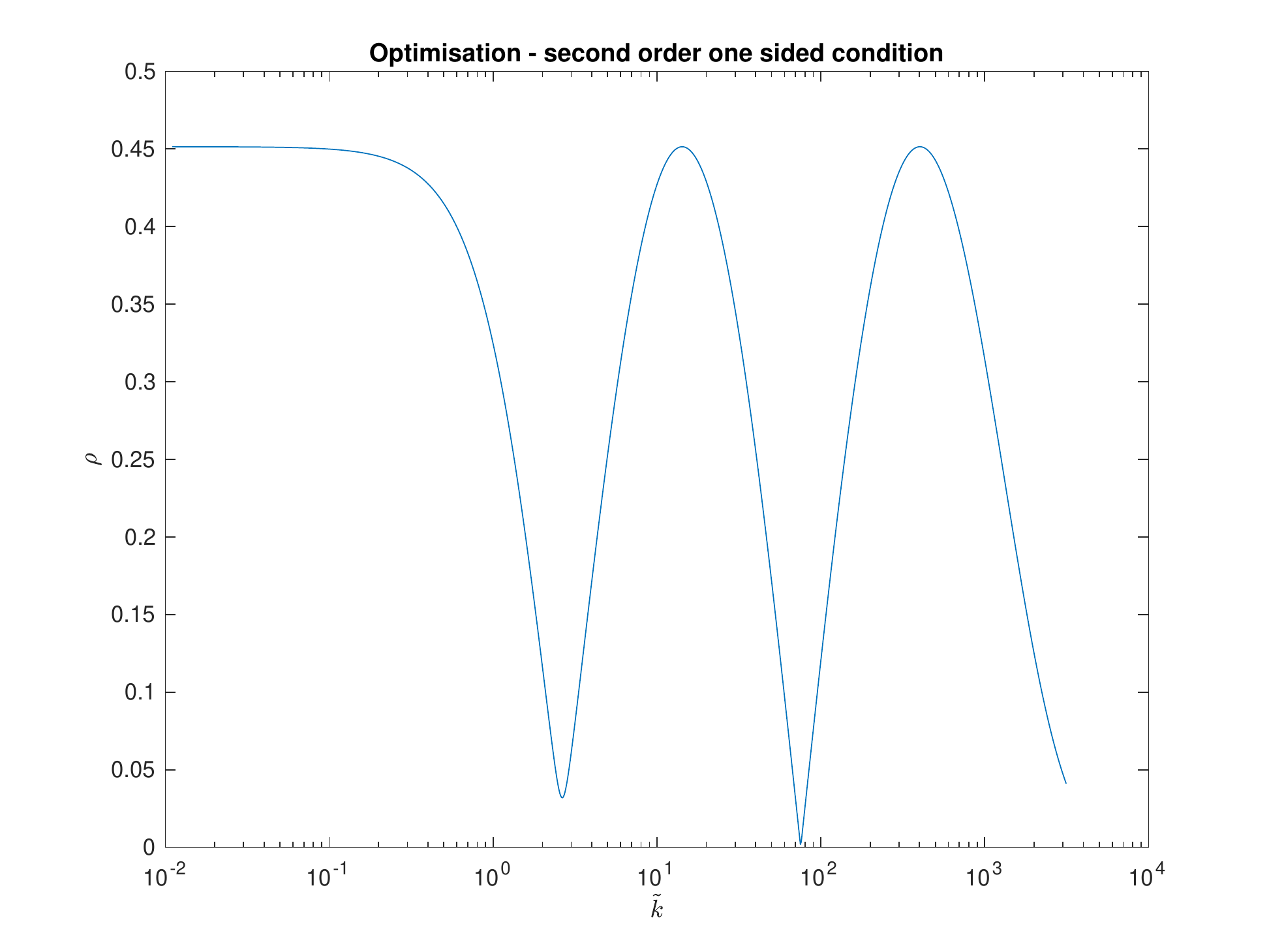}
  \includegraphics[width=0.42\textwidth]{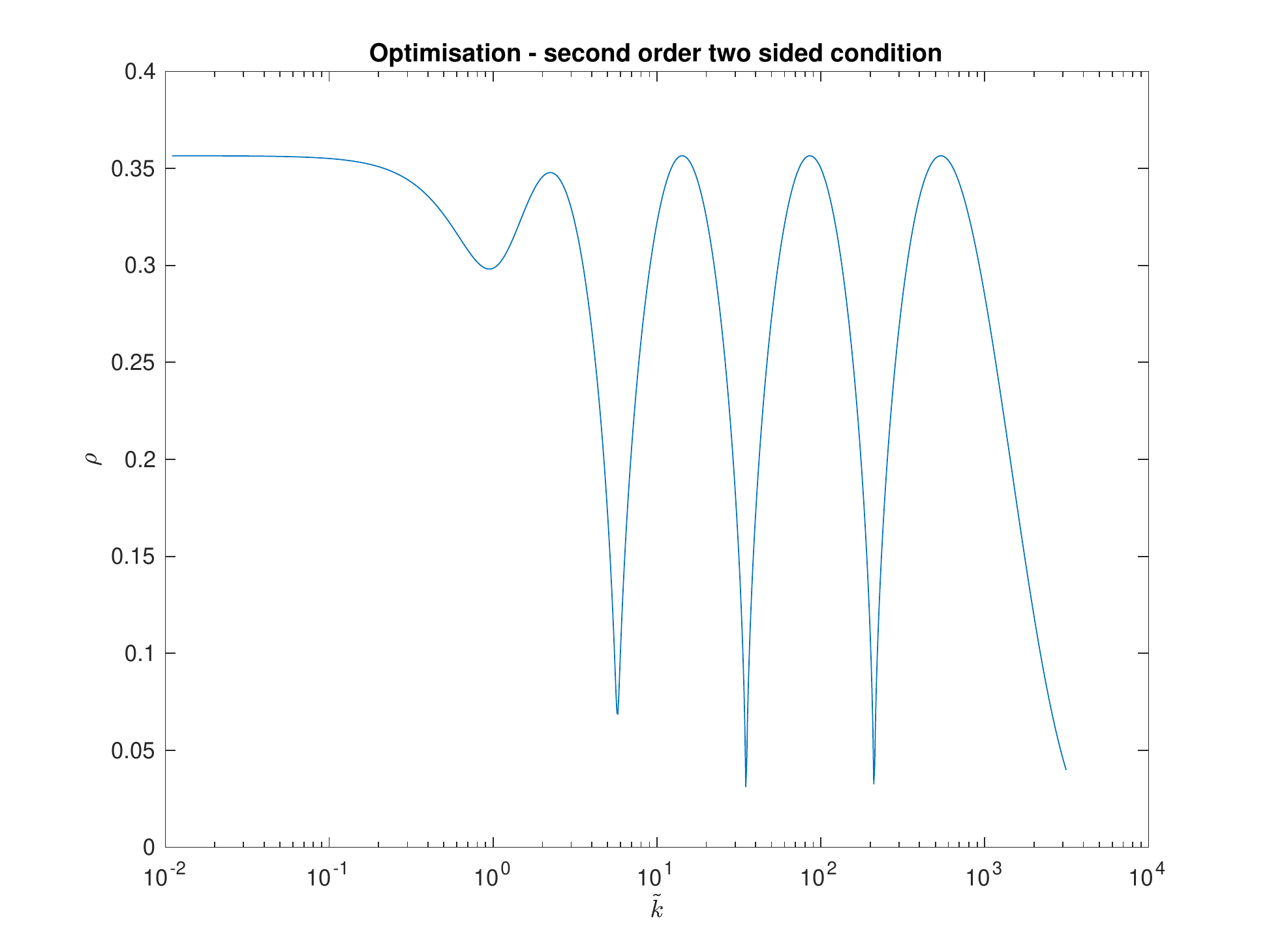}
  \caption{Equioscillation with optimized one sided and two-sided Ventcell
    transmission conditions.}
  \label{fig:2sdopt2nd}
\end{figure}
The main difference compared to Robin transmission conditions is that
we have now more equioscillations.

In the case of one sided conditions we have two equioscillations $\rho(\tilde k_{\min})=\rho(\tilde{k}_1^*) = \rho(\tilde{k}_2^*)$, where $\tilde{k}_j^*$ are two interior local maxima, and we have asymptotically
$$
p \sim C_{p} \delta^{-1/5}, q \sim C_{q} \delta^{3/5}, \rho \sim 1
- C_R \delta^{1/5} + {\cal O}(\delta^{2/5}), \tilde{k}_1^* \sim C_{k_1}
\delta^{-2/5},\, \tilde{k}_2^* \sim C_{k_2} \delta^{-4/5}.
$$
By expanding for $\delta$ small, and setting the leading terms in the
derivatives $\frac{\partial \rho}{\partial k} (\tilde{k}_{1,2}^*)$ to
zero, we get $C_{p}=\frac{2C_{k_1}^2}{C_{k_2}^2}$, $C_{q} =
\frac{2}{C_{k_2}^2}$. Expanding the maxima leads to
$$
\rho(\tilde{k}_1^*) \sim  1 - 8
\frac{C_{k_1}}{C_{k_2}^2} \delta^{1/5} +{\cal O}(\delta^{2/5}),\, 
\rho(\tilde{k}_2^*) \sim  1 - 2 C_{k_2}
\delta^{1/5} +{\cal O}(\delta^{2/5}),
$$
and equating $\rho(\tilde{k}_1^*) = \rho(\tilde{k}_2^*)$ we get $C_{k_1} =
\frac{C_{k_2}^3}{4}$ and $C_R = 2C_{k_2}$.  Finally equating $\rho(\tilde k_{\min}) =
  \rho(\tilde{k}_2^*)$ asymptotically determines uniquely
$C_{k_2} = 2^{3/5} K^{1/5}$ (with $K$ defined in \eqref{eq:K})  and then $C_{k_1} = 2^{-1/5} K^{3/5}$ and $C_{p}=
2^{-3/5} K^{4/5}$, $C_{q}= 2^{-1/5} K^{-2/5}$.

In the case of two-sided conditions, we have four equioscillations,
  $\rho(k_{\min})=\rho(\tilde{k}_1^*) = \rho(\tilde{k}_2^*) = \rho(\tilde{k}_3^*) = \rho(\tilde{k}_4^*) $, where
  $\tilde{k}_j^*$ are four interior local maxima, and we have asymptotically
$$
\begin{array}{c}
p_1 \sim C_{p_1} \delta^{-1/9}, q_1\sim C_{q_1} \delta^{7/9},  p_2\sim C_{p_2} \delta^{-3/9}, q_2 \sim C_{q_2} \delta^{5/9}, \rho \sim 1
- C_R \delta^{1/9} + {\cal O}(\delta^{2/9}),\\
\tilde{k}_1^* \sim C_{k_1} \delta^{-2/9}, \tilde{k}_2^* \sim C_{k_2} \delta^{-4/9}, \tilde{k}_3^*\sim C_{k_3}
\delta^{-6/9}, \tilde{k}_4^* \sim C_{k_4} \delta^{-8/9}.
\end{array}
$$  
By expanding for $\delta$ small, and setting the leading terms in the
  derivatives $\frac{\partial \rho}{\partial k} (\tilde{k}_{1,2,3,4}^*)$ to zero, we get 
\begin{equation}
\label{eq:Cp}
C_{p_1}=\frac{C_{k_1}^2\cdot C_{k_3}^2}{C_{k_2}^2\cdot C_{k_4}^2},\quad C_{p_2}= \frac{C_{k_2}^2\cdot C_{k_4}^2}{ C_{k_3}^2},\quad C_{q_1}=\frac{1}{C_{k_4}^2},\quad C_{q_2} = \frac{C_{k_4}^2}{C_{k_3}^2}.
\end{equation}
Expanding the maxima leads to
$$
\begin{array}{l}
\rho(\tilde{k}_1^*)\sim 1 - 2 \frac{C_{k_1}\cdot C_{k_3}^2}{C_{k_2}^2\cdot C_{k_4}^2} \delta^{1/9} +{\cal O}(\delta^{2/9}),\,
\rho(\tilde{k}_2^*)\sim 1 - 2 \frac{C_{k_2}\cdot C_{k_4}^2}{ C_{k_3}^2\cdot C_{k_4}^2} \delta^{1/9} +{\cal O}(\delta^{2/9}),\\
 \rho(\tilde{k}_3^*)\sim 1 - 2 \frac{C_{k_3}}{C_{k_4}^2} \delta^{1/9} +{\cal O}(\delta^{2/9}),\,
 \rho(\tilde{k}_4^*)\sim 1 - 2 C_{k_4}\delta^{1/9} +{\cal O}(\delta^{2/9}).
\end{array}
$$
Equating now $\rho(\tilde{k}_1^*) = \rho(\tilde{k}_2^*) = \rho(\tilde{k}_3^*) = \rho(\tilde{k}_4^*) $ 
we get $C_{k_1} = C_{k_4}^7, C_{k_2} = C_{k_4}^5,C_{k_3} = C_{k_4}^3$, $C_R = 2C_{k_4}$ and from \eqref{eq:Cp} we get
$$
C_{p_1}= C_{k_4}^8,\, C_{p_2}=\frac{1}{C_{k_4}^2},\, C_{p_3}=C_{k_4}^6,\, C_{p_4}=\frac{1}{C_{k_4}^4}. 
$$
Finally equating $\rho(\tilde k_{\min}) =
  \rho(\tilde{k}_4^*)$ asymptotically determines uniquely
$C_{k_4} = 2^{-1/9} K^{1/9}$ with $K$ given in \eqref{eq:K} and the other constants are determined accordingly.
\end{proof}

\subsection{Optimization for $J$ subdomains}

Like in the case of the Robin conditions, the high frequency behavior
of the convergence factor allows us to study systematically the
asymptotic form of the best parameter choice for $J$ subdomains,
depending on one constant only.
\begin{lemma}[Generic optimized Ventcell asymptotics]\label{GenericLemma2}
  The best choice for one sided Ventcell transmission conditions is
  $p_j^+=p_{j}^-=p^*$, $q_j^+=q_{j}^-=q^*$, and when the overlap
  $\delta$ goes to zero, we have
  \begin{equation}\label{HighFreqMaxVentcell}
    p\sim \frac{C_{k_2}^4}{8} \delta^{-1/5}, \quad q \sim \frac{2}{C_{k_2}^2}\delta^{3/5}, \qquad
    \rho \sim 1 -  2C_{k_2}\delta^{1/5} + {\cal O}(\delta^{2/5}),
  \end{equation}
  where the constant $C_{k_2}$ depends on the number of
  subdomains. For the two-sided Ventcell transmission conditions, the
  best choice is $(p_j^+,q_j^+), (p_{j+1}^-,q_{j+1}^-) \in \{
  (p_1^*,q_1^*),(p_2^*,q_2^*)\}$, $(p_j^+,q_j^+)\ne
  (p_{j+1}^-,q_{j+1}^-) \, \forall j=1..,J-1$, with
  \begin{equation}\label{HighFreqMaxVentcell2}
   (p_1^*,q_1^*) \sim \left(C_{k_4}^8  \delta^{-1/9},\frac{1}{C_{k_4}^2}\delta^{7/9} \right),\,(p_2^*,q_2^*)  \sim \left(C_{k_4}^6  \delta^{-3/9},\frac{1}{C_{k_4}^4}\delta^{5/9} \right),\,
  \rho^* \sim 1 - 2C_{k_4} \delta ^{1/9},
  \end{equation}
  where the constant $C_{k_4}$ depends again on the number of subdomains.
\end{lemma}
\begin{proof}
Similar to the Robin case in Lemma \ref{GenericLemma}, the proof is a
direct consequence of Theorem \ref{Theorem2sub2par} where the high
frequency arguments are identical in the two or more subdomain case.
\end{proof}

It remains to study the constants $C_{k_2}$ and $C_{k_4}$, which are
determined by equioscillation with the low frequency convergence
factor and which depend on the number of subdomains. To simplify the
computations, we assume again Dirichlet boundary conditions at the
outer boundaries of the global domain and we start by computing the
leading order terms in $T$ for small overlap $\delta$.  For one sided
conditions $p^* = \frac{C_{k_2}^4}{8} \delta^{-1/5}$, $q^* =
\frac{2}{C_{k_2}^2}\delta^{3/5} $ we get, with
  $s:=\sqrt{\tilde{k}_{\min}^2+\eta - i\varepsilon}$,
$$
\begin{array}{c}
\displaystyle \alpha_j^+(\tilde k_{\min}) = \alpha_j^-(\tilde k_{\min}) = \frac{4 s e^{-sL}}{C_p (1- e^{-2sL} )} \delta^{1/5} =  \frac{32 s e^{-sL}}{C_k^5 (1- e^{-2sL} )} \delta^{1/5} := \tilde a,\\
\displaystyle \beta_j^+(\tilde k_{\min}) = \beta_j^-(\tilde k_{\min}) = 1 - \frac{2 s( e^{-2sL} + 1)}{C_p (1- e^{-2sL} )}\delta^{1/5} = 1 - \frac{16 s( e^{-2sL} + 1)}{C_k^4 (1- e^{-2sL} )}\delta^{1/5}=: \tilde b,
\end{array}
$$
leading to the low frequency iteration matrix of the same form like in \Cref{eq:tlf1}.
By computing the spectral radius of this matrix for $J=2,3,4,\ldots$
subdomains, we get for small overlap $\delta$
\begin{equation}\label{eq:const2}
  \rho_J(\tilde{k}_{\min})  \sim  1 - \frac{16}{C_{k_2}^4} K_J  \delta^{1/3},
\end{equation}
with the same constant $K_J$ defined in \eqref{KJ}, and we obtain
$\rho_J(\tilde{k}_{\min})\sim
1-\frac{16K_J}{C_{k_2}^4}\delta^{1/5}$. By equating this with the high
frequency maximum $\rho(k^*)\sim 1 - 2C_{k_2} \delta ^{1/5}$ from
\eqref{HighFreqMaxVentcell} leads to $C_{k_2}=(8K_J)^{1/5}$.

For the two-sided Ventcell transmission conditions with $(p_1^*,q_1^*)
\sim ({C_{k_4}^8} \delta^{-1/9}, \frac{1}{C_{k_4}^2}\delta^{7/9})$,
$(p_2^*,q_2^*) = ({C_{k_4}^6} \delta^{-3/9},
\frac{1}{C_{k_4}^4}\delta^{5/9})$ we have
$$
\begin{array}{c}
\displaystyle \alpha_j^+(\tilde k_{\min}) = \alpha_j^-(k_{\min}) = \frac{2 s e^{-sL}}{C_{p_2} (1- e^{-2sL} )} \delta^{1/9} =  \frac{2 s e^{-sL}}{C_{k_4}^4 (1- e^{-2sL} )} \delta^{1/9} := \tilde a,\\
\displaystyle \beta_j^+(\tilde k_{\min}), \beta_{j+1}^-(k_{\min}) \in \{ \delta^{2/9} C_{k_4}^2 \tilde b, \frac{1}{\delta^{2/9} C_{k_4}^2} \tilde b\},\quad  \tilde b:=  1 - \frac{s( e^{-2sL} + 1)}{C_{k_4}^8 (1- e^{-2sL} )}\delta^{1/9},
\end{array}
$$
leading to the low frequency iteration matrix of the same form like in
\Cref{eq:tlf2} where again the couples $ \tilde b_+ \ne \tilde b_-$ can
vary along the diagonal but still lay in the set $\{ \delta^{2/9}
C_{k_4}^2 \tilde b, \frac{1}{\delta^{2/9} C_{k_4}^2} \tilde b\}$ which
does not change the eigenvalues of the matrix.  By computing the
spectral radius of this matrix for $J=2,3,4,\ldots$ subdomains we get
for small overlap $\delta$
\begin{equation}
  \rho_J(k_{\min}) \sim 1 - \frac{K_J}{C_{k_4}^8}\delta^{1/9}
\end{equation}
with the same constant $K_J$ from \Cref{KJ}, and equating with
$\rho(k^*) \sim 1 - 2C_{k_4} \delta ^{1/9}$ from
\eqref{HighFreqMaxVentcell2} we obtain
$C_{k_4}=\left(\frac{K_J}{2}\right)^{1/9}$.  We therefore get,
using Lemma \ref{GenericLemma2}, the following result for the $J$
subdomain decomposition:
\begin{theorem}[$J$ Subdomain Ventcell Optimization]
For $J$ subdomains, the best choice in the one sided Ventcell
transmission conditions is $p_j^+=p_{j}^-=p^*,\,q_j^+=q_{j}^-=q^*$,
and when the overlap $\delta$ goes to zero, we have
  \begin{equation}
     p^*\sim 2^{-3/5}K_J^{4/5} \delta^{-1/5}, \, q^* \sim 2^{-1/5}K_J^{-2/5} \delta^{3/5}, \quad
    \rho \sim 1 -  2^{8/5}K_J^{1/5} \delta^{1/5} + {\cal O}(\delta^{2/5}),
  \end{equation}
  with the constant $K_J$ from \Cref{KJ}. For the two-sided Ventcell
  transmission conditions the best choice is $(p_j^+,q_j^+),
  (p_{j+1}^-,q_{j+1}^-) \in \{ (p_1^*,q_1^*),(p_2^*,q_2^*)\}$,
  $(p_j^+,q_j^+)\ne (p_{j+1}^-,q_{j+1}^-) \, \forall j=1..,J-1$, with
  \begin{equation}
   (p_1^*,q_1^*) \sim \left(2^{-8/9} K^{8/9} \delta^{-1/9}, 2^{2/9} K^{-2/9} \delta^{7/9}\right),\,
   (p_2^*,q_2^*)  \sim \left(2^{-2/3} K^{2/3} \delta^{-1/3}, 2^{4/9} K^{-4/9} \delta^{5/9} \right),
  \end{equation}
  leading to $\rho^* \sim 1 - 2^{8/9} K_J^{1/9} \delta ^{1/9}$.
\end{theorem}

Since the asymptotic convergence factors for Ventcell transmission
  conditions depend on the same constant $K_J$ as for Robin
  transmission conditions, the one-level methods with Ventcell
  transmission conditions have the same scalability properties
  described in Subsection \ref{ScalabilitySubSec} for Robin
  transmission conditions, and we will illustrate this as well in the
  following section with numerical experiments.

\section{Numerical results}


For our tests, we start with a decomposition into four overlapping
domains, that can be uniform (four rectangles) or a more general
decomposition from METIS as shown in Figure
\ref{eq:4subdomains}.
\begin{figure}
\centering
\includegraphics[width=0.48\textwidth]{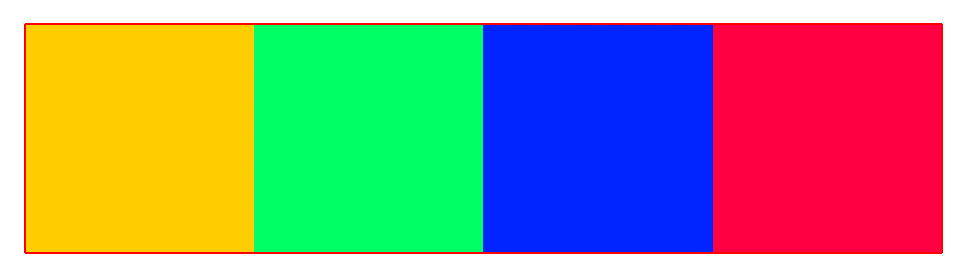}
\includegraphics[width=0.48\textwidth]{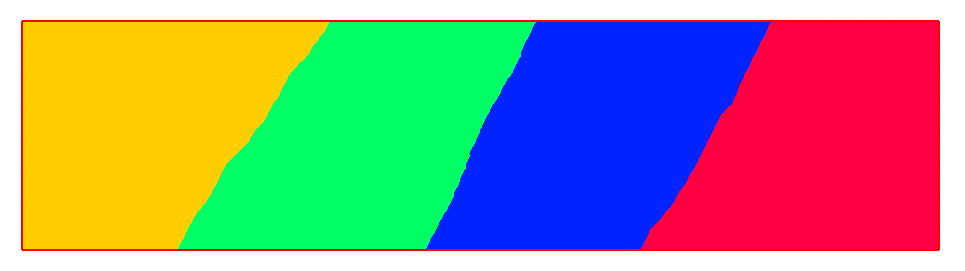}
\caption{Decomposition into four subdomains: uniform (left) and
  METIS (right).}
 \label{eq:4subdomains}
\end{figure}
Throughout this section we fix the values of the
parameters to $\eta =\varepsilon = 1$. To discretise we use a uniform
square grid in each direction and triangulate to form P1 elements. We
use an overlap of size $\delta=2h$, with $h$ being the mesh size. All
computations are performed using FreeFem (\url{http://freefem.org/}).

\subsection{Optimised Schwarz methods as solvers: asymptotic behaviour}

In the first series of tests we consider one and two-sided Robin and
Ventcell transmission conditions, and we increase locally the number
of degrees of freedom, which leads to a decreasing value of the mesh
size $h$, and thus the overlap $\delta$. We show results for the
iterative version of the classical Restricted Additive Schwarz method
(RAS) and the optimised versions of the algorithm using their ORAS
(Optimized Restricted Additive Schwarz) implementation
\cite{gander2008schwarz}. In Table \ref{tab:comp1} we report the
iteration count in order to achieve a relative discrete
$L^2$-norm error reduction of $10^{-6}$ (the numbers in
parentheses correspond to the METIS decompositions). In all tests we
start with a random initial guess in order to ensure that all
frequencies are present in the error.
\begin{table}
 \centering
\begin{tabular}{|c|c|c|c|c|c|}
\hline
$h$                   & RAS                & Robin 1 & Robin 2 & Ventcell 1 & Ventcell 2 \\ \hline
$\frac{1}{50}$   &    53 (74)         & 11 (12)            & 10 (16)            &   7 (8)                 &  9 (11)\\ 
$\frac{1}{100}$ &  105 (132)       & 13 (15)           & 12 (20)            &   8 (8)                  &  9 (12) \\ 
$\frac{1}{200}$ &  207 (270)      & 17 (19)           & 14 (18)            &  9 (10)                 & 10 (13) \\ 
$\frac{1}{400}$ &  412 (508)       & 22 (25)           & 16 (23)            & 11 (11)                & 11 (15) \\ 
$\frac{1}{800}$ &  820 (960)       & 27 (32)           & 20 (26)            & 12 (14)               & 13 (16) \\ 
$\frac{1}{1600}$ & 1650 (1810)  & 32 (39)          & 23 (29)             & 14 (17)              & 14 (17) \\ 
 \hline
\end{tabular}
\caption{RAS vs. One and two-sided Robin and Ventcell conditions for refined meshes}
  \label{tab:comp1}
\end{table}
We see that in the case of the classical Schwarz algorithm (RAS) the
iteration count increases linearly when the mesh size representing the
overlap is decreased linearly, leading to an important number of
iterations for very fine meshes. This statement is true both in the
case of uniform and METIS decompositions, and the iteration count for
the latter is slightly larger. These results can be greatly improved
by using optimised versions of the algorithm, and we notice a
progressive improvement of the behaviour, first with Robin and then
with Ventcell transmission conditions, as predicted by our analysis.

To see the asymptotic behavior, we plot the iteration counts in Figure
\ref{fig:asymiter}.
\begin{figure}
\centering
\includegraphics[width=0.48\textwidth]{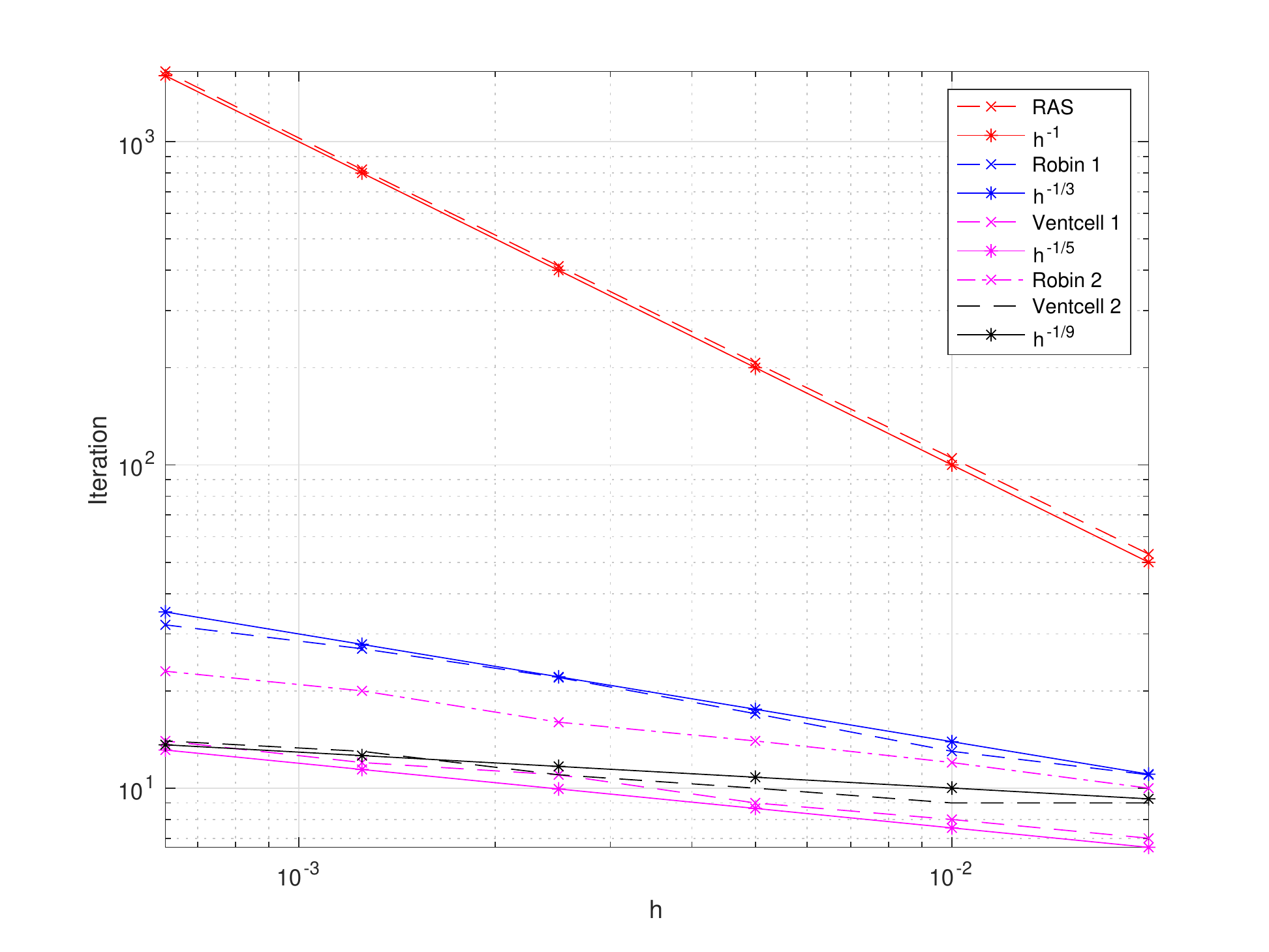}
\includegraphics[width=0.48\textwidth]{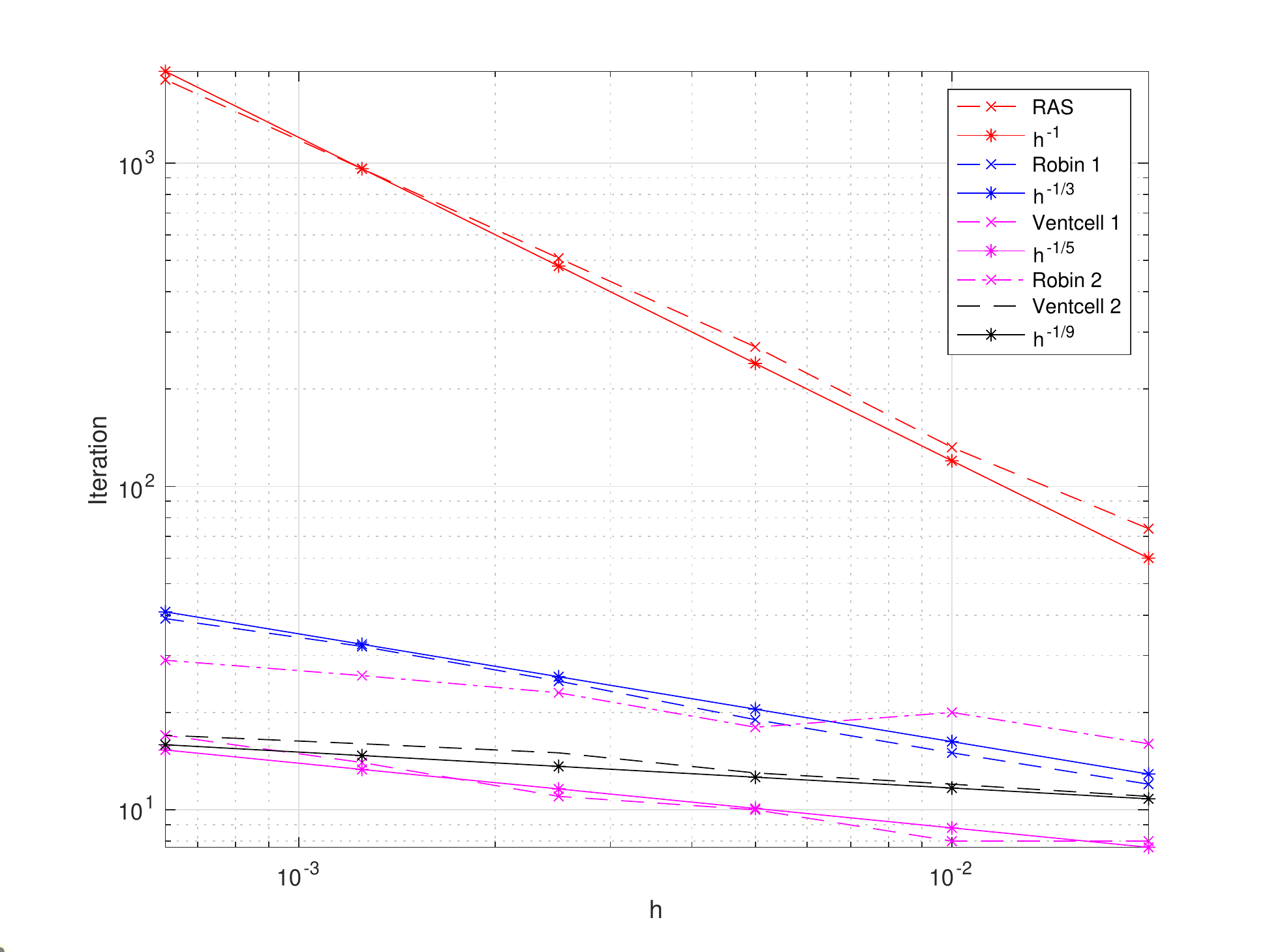}
\caption{Iteration count depending on the mesh size for classical
  Dirichlet (RAS), Robin and Ventcell transmission conditions.}
 \label{fig:asymiter}
\end{figure}
We see that the results are consistent with the theory, i.e. in the
case of Robin 1 (one-sided) transmission conditions, the iteration
count increases like $h^{-1/3}$ and in the case of Ventcell 1 like
$h^{-1/5}$. For their two-sided variants the iteration count increases
like $h^{-1/5}$ for Robin 2 and like $h^{-1/9}$ for Ventcell 2. This
behaviour holds both for uniform and METIS decompositions. Note that
while the asymptotic behaviour for two-sided Robin and one-sided
Ventcell transmission conditions is the same, the iteration count is
different, since the constants involved in the theoretical estimates
are different. Also, in order to see the full benefit of two-sided
Ventcell (Ventcell 2) conditions, highly refined meshes are needed
leading to sufficiently large problems.

\subsection{Optimised Schwarz methods as solvers: scalability}\label{NumSubSecScalability}
  
We now perform the same kind of experiments but with an increasing
number of subdomains and a strip-wise decomposition where we consider
again the case of uniform and METIS decompositions.  We keep the size
of the local subdomains fixed and we choose $h=1/100$, see Table
\ref{tab:comp2}.
\begin{table}
 \centering
\begin{tabular}{|c|c|c|c|c|c|}
\hline
$J$   & RAS                & Robin 1 & Robin 2 & Ventcell 1 & Ventcell 2 \\ \hline
$2$   &  97 (127)        &    13 (15)        &    12 (14)        &     7 (9)     &   9 (10) \\ 
$4$   &  105 (132)       &     13 (15)      &   12 (20)        &    8 (8)    &  9 (12)  \\ 
$8$   &  107 (141)      &    14 (15)        &     12 (14)         &     8   (9)     &  10 (11)\\ 
$16$ &  108 (132)    &      14 (15)       &       12 (25)      &      8   (9)      & 10 (12)  \\ 
$32$ &    108 (138)     &    14 (15)       &       12 (25)       &     8   (9)     &  10 (12) \\ 
 \hline
\end{tabular}
\caption{RAS vs. one and two-sided Robin and Ventcell conditions for a
  strip-wise decomposition into $J$ subdomains (fixed subdomain
  size).}
  \label{tab:comp2}
\end{table}
In the case of the classical Schwarz method (RAS), we notice that
after a slight increase in iterations when the number of subdomains
grows, the iteration count stabilises. This is consistent with the
theoretical results and shows that the one-level method is weakly
scalable in this setting: the iteration count remains constant when
the number of subdomains is increased and the size of the subdomains
is kept fixed. The optimised one-level variants are scalable as well,
as shown in Subsection \ref{ScalabilitySubSec}, the iteration counts
are much lower, and remain constant almost from the very beginning
when the number of subdomains increases. Note that we cannot control
the exact size of the subdomains in the case of METIS decompositions,
which explains the slight variations in the iteration counts for METIS
decompositions.
  

In a second series of tests, we keep now the size of the global domain
fixed to $[0,1]^2$, and chose the mesh size equal to $h = 1/512$. We
increase the number of subdomains in one direction in order to obtain
a strip-wise decomposition. In this case, the domains will become
thinner and thinner as shown in Figure \ref{eq:thinsub}.
\begin{figure}
\centering
\includegraphics[width=0.24\textwidth]{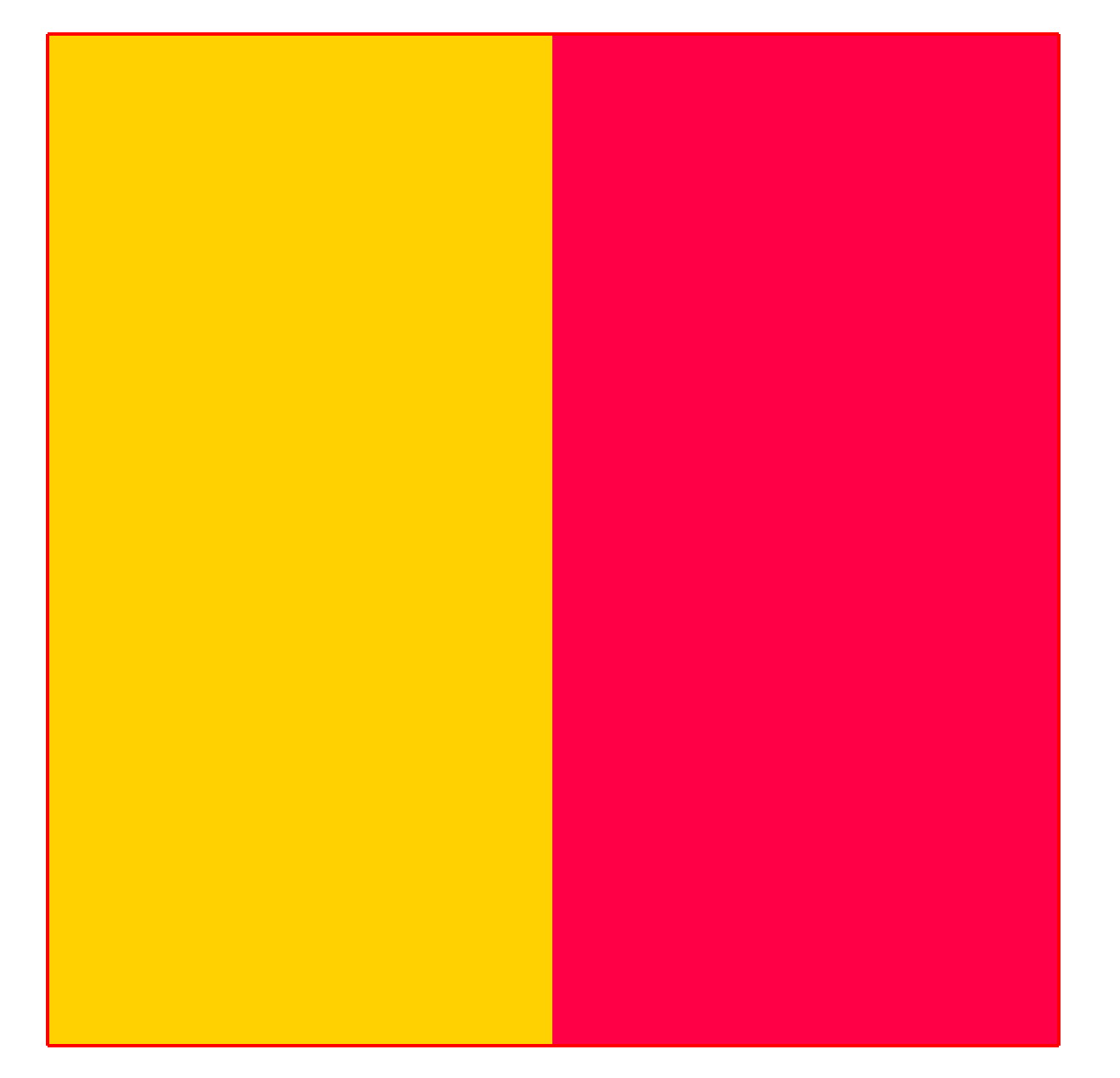}
\includegraphics[width=0.24\textwidth]{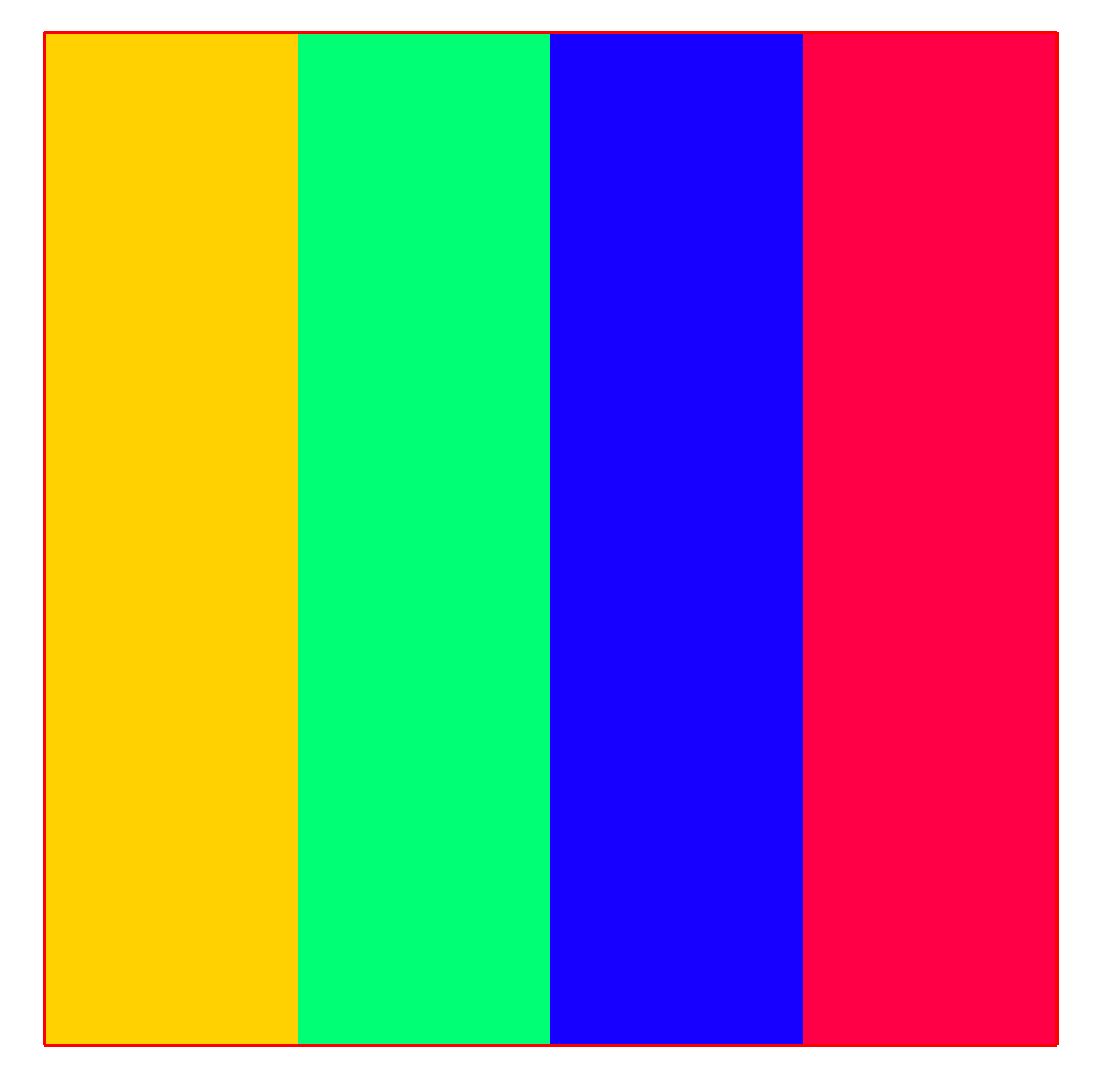}
\includegraphics[width=0.24\textwidth]{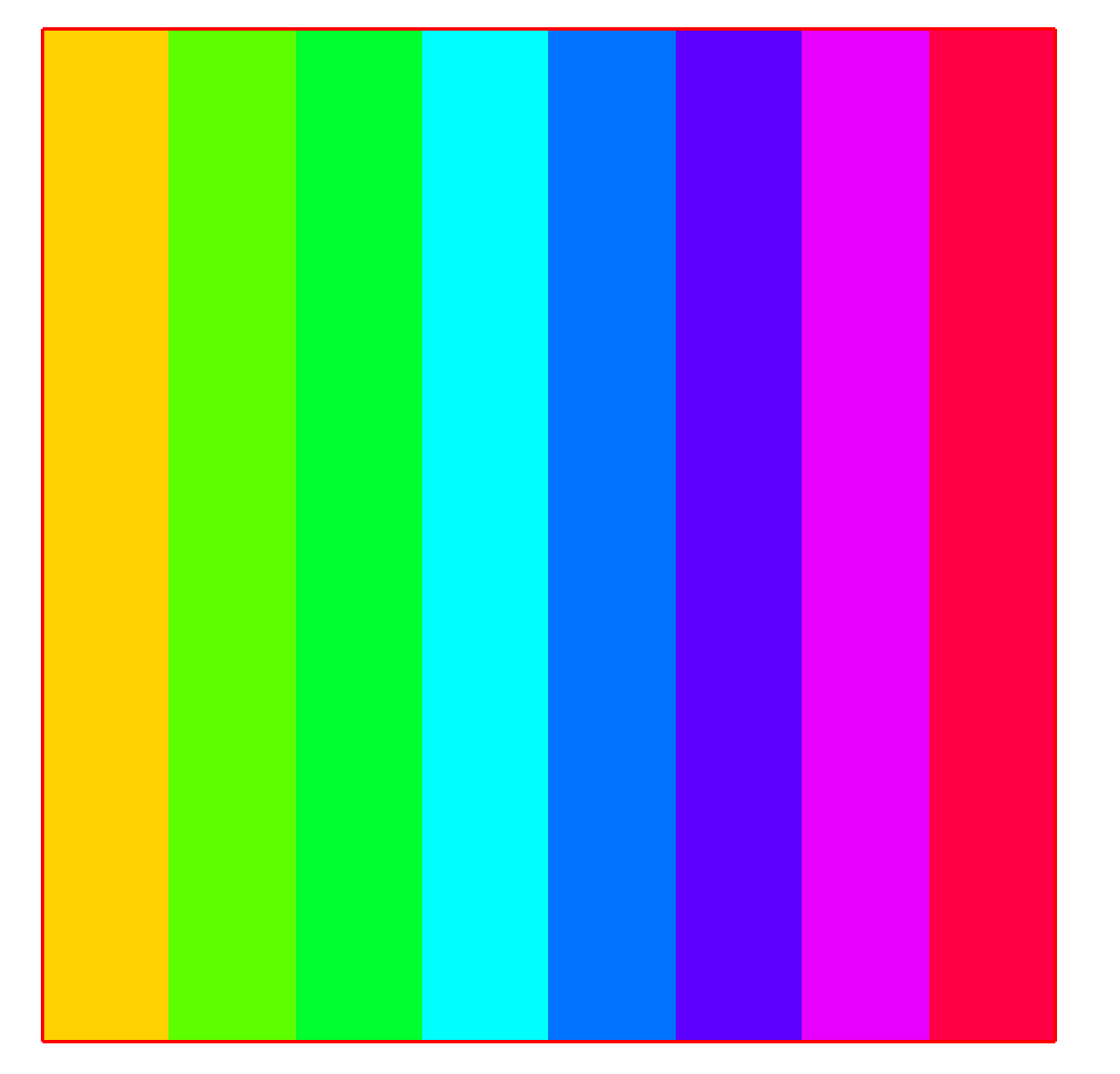}
\includegraphics[width=0.24\textwidth]{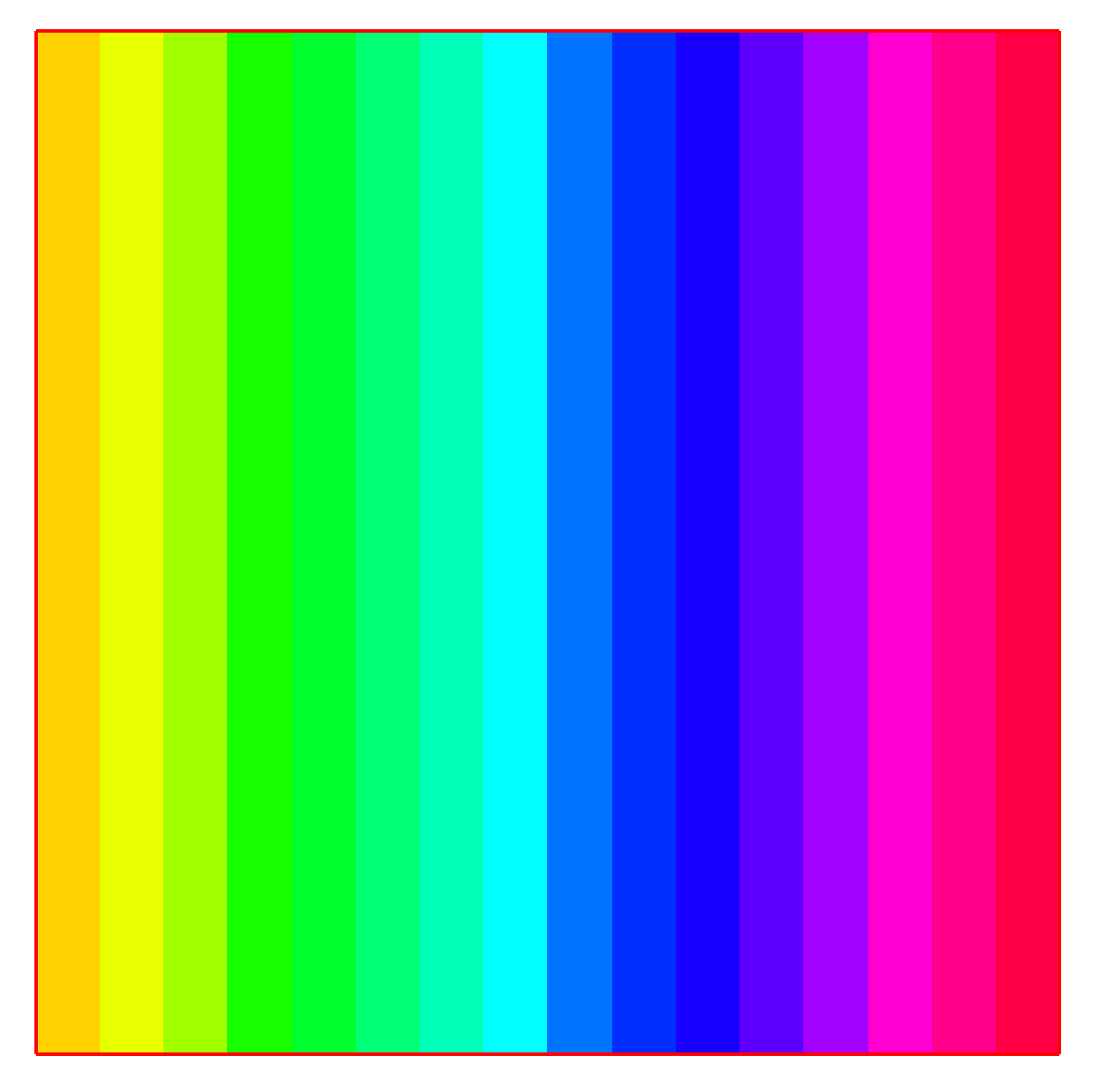}
\caption{Decomposition into many subdomains of decreasing width.}
 \label{eq:thinsub}
\end{figure}
We see in Table \ref{tab:comp3}
\begin{table}
 \centering
\begin{tabular}{|c|c|c|c|c|c|}
\hline
$J$   & RAS                & Robin 1 & Robin 2 & Ventcell 1 & Ventcell 2 \\ \hline
$2$   &    466         &    22    &     17     &   11    &  12  \\ 
$4$   &     731        &    32    &     27     &    12   &  25  \\ 
$8$   &  $> 1000$  &    56    &     -   &  28    &  -\\ 
$16$ &  $> 1000$  &  103    &     - &      62      & -  \\ 
$32$ &  $> 1000$  &  196    &     -   &    132      & - \\ 
 \hline
\end{tabular}
\caption{RAS vs. one and two-sided Robin and Ventcell conditions for a
  strip-wise decomposition into $J$ subdomains with decreasing
  subdomain size.}
  \label{tab:comp3}
\end{table}
that these one level methods are no longer scalable, as predicted in
Subsection \ref{ScalabilitySubSec}. The iteration count for RAS grows
rapidly with the number of domains to exceed 1000 iterations. The
iteration numbers for Robin 1 and Ventcell 1 are much lower, but also
approximately double when the number of subdomains doubles, as
predicted by the theoretical constants $K_J$ in \eqref{KJ} and
$K_{\infty}$ in \eqref{eq:cnr} when the subdomain size $L$ becomes
small. Furthermore, when two-sided conditions are used, the method is
no longer convergent as iterative solver when we increase the number
of subdomains in this setting, which changes however when using the
methods as preconditioners in the next section.

\subsection{Optimised Schwarz methods as preconditioners}
\label{NumSubSecGMRES}

We solve the discretised problem using GMRES where the parallel
Schwarz method with Robin or Ventcell conditions is used as a
preconditioner. In particular, we use right-preconditioned GMRES and
terminate when a relative residual tolerance of $10^{-6}$ is
reached. The preconditioner, which arises naturally as the discretised
version of the parallel Schwarz method we have studied, is known as
the one-level optimised restricted additive Schwarz (ORAS)
preconditioner. This ORAS preconditioner is given by
\begin{align*}
\mathbf{M}^{-1} = \sum_{i=1}^{N} \mathbf{R}_i^T \mathbf{D}_i \mathbf{\tilde{A}}_i^{-1} \mathbf{R}_i,
\end{align*}
where $\left\{\mathbf{R}_i\right\}_{1 \le i \le N}$ are the Boolean
restriction matrices from the global to the local finite element
spaces and $\left\{\mathbf{D}_i\right\}_{1 \le i \le N}$ are local
diagonal matrices representing a partition of unity. The key
ingredient of the ORAS method is that the local subdomain matrices
$\{\mathbf{\tilde{A}}_i\}_{1 \le i \le N}$ incorporate more
efficient Robin or Ventcell transmission conditions.

We now perform exactly the same kind of experiments as in the previous
subsections, but test the performance of the preconditioner instead of
the stationary iterative solver. We start by the strip-wise
decomposition into four subdomains like in Figure
\ref{eq:4subdomains}, and we refine the mesh locally in each
subdomain. In Table \ref{tab:comp4}
\begin{table}
 \centering
\begin{tabular}{|c|c|c|c|c|c|}
\hline
$h$                   & RAS            & Robin 1 & Robin 2 & Ventcell 1 & Ventcell 2 \\ \hline
$\frac{1}{50}$   &  15 (21)       & 8 (8)      &    9 (11)     &      6 (6)    & 7 (8)  \\ 
$\frac{1}{100}$ &  21 (28)       & 9 (10)    &   10 (13)   &      6 (7)     &  7 (9) \\ 
$\frac{1}{200}$ &  28 (41)       & 10 (11)   &   11 (13)   &      7 (7)    &  7 (9) \\ 
$\frac{1}{400}$ &   39 (57)      & 12 (14)   &   13 (15)   &      8 (8)    &  9 (10)\\ 
$\frac{1}{800}$ &  56 (82)       &  13 (15)  &    14 (17)   &     9 (9)    & 9 (11)  \\ 
$\frac{1}{1600}$ &  80 (112)  &   15 (17)   &   15 (17)   &     10 (11)  &  9 (11)\\ 
 \hline
\end{tabular}
\caption{GMRES preconditioned by RAS and ORAS with Robin and Ventcell
  conditions for refined meshes.}
  \label{tab:comp4}
\end{table}
we report the iteration count of preconditioned GMRES in order to
achieve a relative discrete $L^2$-norm error reduction of $10^{-6}$ (the numbers in
parentheses correspond to METIS decompositions). Again we start with a
random initial guess in order to ensure that all frequencies are
present in the error. We see that the iteration count is less
sensitive to the choice of transmission conditions, but the hierarchy
of the methods is preserved, and iteration counts depend for the
optimized methods only very weakly on the mesh size that represents
the overlap. We repeat next the weak scaling experiments from Tables
\ref{tab:comp2} and \ref{tab:comp3}, where the methods are used as
preconditioners. We see in Table \ref{tab:comp5}
\begin{table}
 \centering
\begin{tabular}{|c|c|c|c|c|c|}
\hline
$J$   & RAS          & Robin 1 & Robin 2 & Ventcell 1 & Ventcell 2 \\ \hline
$2$   &    17 (21)   &   8 (9)     &  9 (11)   &  6 (7)   &   6 (8) \\ 
$4$   &   21 (28)    &   9 (10)   &  10 (13) &  6 (7)   &   7 (9)  \\ 
$8$   &   24 (31)    &   9 (10)   &  10 (11)  &  6 (6)   &   7 (8)\\ 
$16$ &    24 (29)   &    9 (10)   &  10 (15)  &  6 (6)  &   7 (10)\\ 
$32$ &    24 (32)   &    9 (10)   &  10 (13)  &  6 (7)  &   7 (8) \\ 
 \hline
\end{tabular}
\caption{GMRES preconditioned by RAS and ORAS with Robin and Ventcell
  conditions for a strip-wise decomposition into $J$ subdomains (fixed
  subdomain size).}
  \label{tab:comp5}
\end{table}
that also the preconditioners scale very well with the increase of the
number of subdomains when the size of the subdomains is kept fixed.
However, when the size of the global domain is fixed and the subdomain
size decreases when their number increases, we see in Table
\ref{tab:comp6}
\begin{table}
 \centering
\begin{tabular}{|c|c|c|c|c|c|}
\hline
$J$   & RAS  & Robin 1 & Robin 2 & Ventcell 1 & Ventcell 2 \\ \hline
$2$   &  39  &   12  &   13   &   7   &  9  \\ 
$4$   &  48  &   15  &  19   &    9  &   12 \\ 
$8$   &  62  &   20  &   24   &   22  &  20 \\ 
$16$ &  79  &   40  &  41  &   50   &  48 \\ 
$32$ & 110 &   81  &  76   &   102  &  100 \\ 
 \hline
\end{tabular}
\caption{RAS vs. One and two-sided Robin and Ventcell conditions for a strip-wise decomposition into $J$ subdomains (decreasing subdomain size).}
  \label{tab:comp6}
\end{table}
that even though all methods are convergent, they are again not
scalable any more, as expected from our analysis, see Subsection
\ref{ScalabilitySubSec}. We also see that in this setting of thinner
and thinner subdomains, the difference of performance between the
classical and optimised preconditioners is becoming less and less
pronounced.

\subsection{General decomposition into subdomains}

We next test our new optimized Schwarz methods also in a setting for
which we do not yet have a convergence analysis, namely two
dimensional uniform
and METIS decompositions into 4 subdomains including cross points,
like in Figure \ref{eq:16sub}.
 \begin{figure}
\centering
\includegraphics[width=0.4\textwidth]{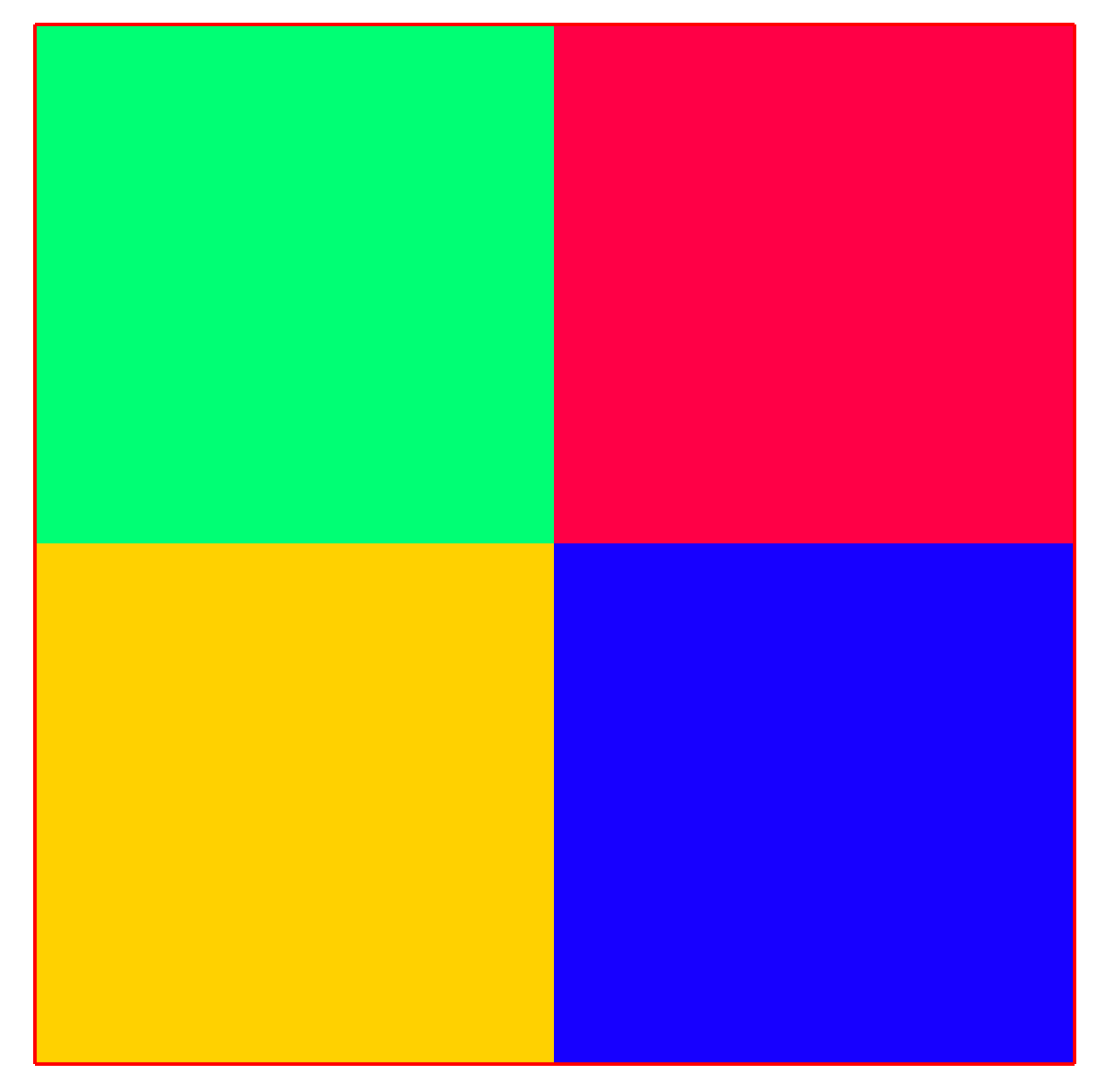}
\includegraphics[width=0.4\textwidth]{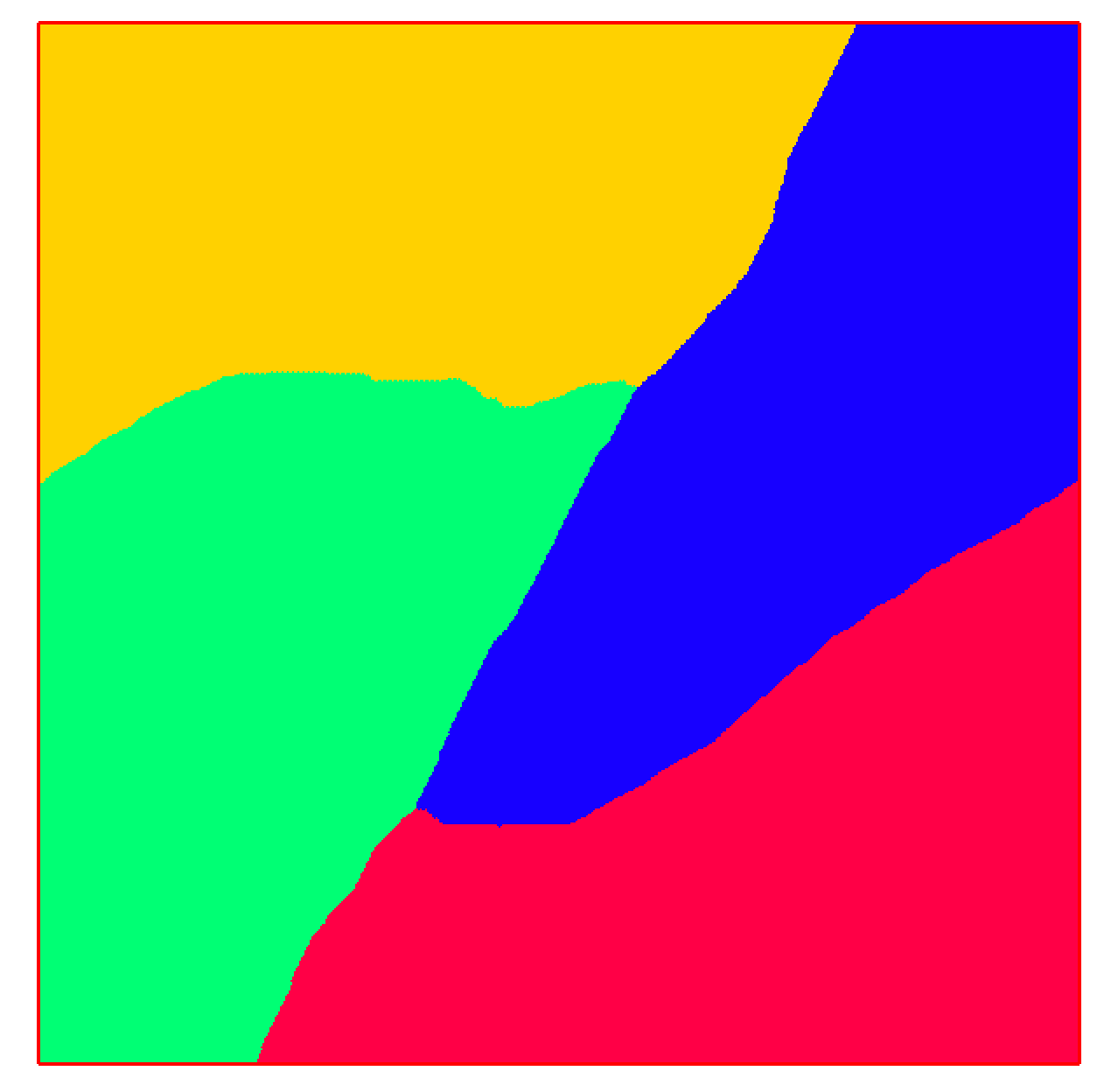}
\caption{Decomposition into 16 subdomains (uniform and METIS).}
 \label{eq:16sub}
\end{figure}
We show the iteration counts needed by the various preconditioned
GMRES methods in Table \ref{tab:comp8}.
\begin{table}
 \centering
\begin{tabular}{|c|c|c|c|c|c|}
\hline
$h$                   & RAS       & Robin 1 & Robin 2 & Ventcell 1 & Ventcell 2 \\ \hline
$\frac{1}{50}$   &  15 (19)  & 14 (13)  &  16 (15) & 12 (13)  &   13 (14)\\ 
$\frac{1}{100}$ &  22 (29)  & 17 (19)  &  20 (23) & 14 (14)  &   14 (16) \\ 
$\frac{1}{200}$ &  30 (41)  & 19 (19)  &  22 (25) & 16 (16)  &   17 (18)\\ 
$\frac{1}{400}$ &  44 (62)  & 23 (23)  &  28 (30)    & 18 (17)  &  19 (20)\\ 
$\frac{1}{800}$ &  63 (88)  & 28 (29)  &  35 (35)   &  20 (20)  &   21 (22)\\ 
 \hline
\end{tabular}
\caption{RAS vs. one and two-sided Robin and Ventcell conditions used
  as preconditioners for GMRES for refined meshes.}
  \label{tab:comp8}
\end{table}
We see that again the optimized variants perform much better than
classical RAS, and also the asymptotic dependence on the mesh size
representing the overlap is much weaker, so that for larger and larger
problems the gain in lower iteration counts is becoming more and more
substantial. For example for mesh size $h=\frac{1}{800}$, Robin
transmission conditions reduce the iteration count by more than a
factor two compared to RAS, and Ventcell transmission conditions by a
factor three, at the same cost per iteration.

\section{Conclusions}

We have shown for the first time that it is possible to optimize
  transmission conditions for many subdomain decompositions in
  optimized Schwarz methods. To do so, an essential ingredient was an
  asymptotic approximation of the convergence factor, and our analysis
  allowed us to precisely characterize the convergence dependence on
  the number of subdomains. Using a new technique of limiting spectra,
  we could even study the case of an infinite number of subdomains,
  leading to a new proof of scalability of these methods for specific
  strip decompositions for general complex diffusion problems. We also
  optimized for the first time two-sided Ventcell transmission
  conditions, which led to a new and much weaker asymptotic dependence
  on the overlap size than all earlier know optimized Schwarz methods.
  We illustrated our theoretical results with numerical experiments,
  including cases not covered by our analysis.

  \bibliographystyle{abbrv}
  \bibliography{paper}

\end{document}